\documentclass[11pt]{article}

\pdftrailerid{}

\usepackage[letterpaper,margin=1in]{geometry}
\usepackage[T1]{fontenc}
\usepackage{lmodern}
\usepackage{microtype}
\usepackage{mathtools,amssymb,amsthm}
\usepackage{etoolbox}
\usepackage{aliascnt}
\usepackage{enumitem}
\usepackage{array,longtable}
\usepackage{fancyhdr}
\usepackage{tikz}
\usetikzlibrary{arrows.meta}
\usepackage[dvipsnames]{xcolor}
\usepackage[normalem]{ulem}
\definecolor{InternalLink}{HTML}{1E63A8}
\definecolor{ExternalLink}{HTML}{167C78}
\newif\ifshowglossary
\showglossaryfalse
\usepackage[
  backend=biber,
  style=alphabetic,
  sorting=nyt,
  maxbibnames=99,
  giveninits=true,
  doi=true,
  eprint=true,
  url=true,
  isbn=true
]{biblatex}
\usepackage[
  colorlinks=true,
  linkcolor=InternalLink,
  citecolor=InternalLink,
  urlcolor=ExternalLink,
  pdfauthor={Jin-Cheng Guu},
  pdftitle={A Disproof of Santharoubane's Conjecture on Presentations of Generic Skein Algebras}
]{hyperref}
\usepackage{zref-user}
\usepackage[nameinlink,noabbrev]{cleveref}

\let\mathscopeoriginalhyperlink\hyperlink
\ExplSyntaxOn
\RenewDocumentCommand{\hyperlink}{mm}
  {
    \tl_if_in:nnTF {#1} {mathscope.}
      {
        \ifshowglossary
          \begingroup
          \hypersetup{linkcolor=black!78}
          \mathscopeoriginalhyperlink{#1}{#2}
          \endgroup
        \else
          #2
        \fi
      }
      {\mathscopeoriginalhyperlink{#1}{#2}}
  }
\ExplSyntaxOff

\newif\ifshowrevisions
\showrevisionsfalse
\newcommand{\revadd}[1]{%
  \ifshowrevisions\textcolor{red!80!black}{#1}\else#1\fi}
\newcommand{\revdel}[1]{%
  \ifshowrevisions\textcolor{gray}{\sout{#1}}\fi}
\newcommand{\revdelnostrike}[1]{%
  \ifshowrevisions\textcolor{gray}{#1}\fi}
\newenvironment{revaddblock}
  {\ifshowrevisions\color{red!80!black}\fi}
  {}

\newtheorem{maintheorem}{Theorem}

\newaliascnt{lemma}{theorem}
\newtheorem{lemma}[lemma]{Lemma}
\aliascntresetthe{lemma}
\newaliascnt{fact}{theorem}
\newtheorem{fact}[fact]{Fact}
\aliascntresetthe{fact}
\newtheorem*{localfact}{Fact}
\newaliascnt{conjecture}{theorem}
\newtheorem{conjecture}[conjecture]{Conjecture}
\aliascntresetthe{conjecture}
\theoremstyle{remark}
\newaliascnt{remark}{theorem}
\newtheorem{remark}[remark]{Remark}
\aliascntresetthe{remark}
\theoremstyle{definition}
\newaliascnt{definition}{theorem}
\newtheorem{definition}[definition]{Definition}
\aliascntresetthe{definition}
\newtheorem*{localdefinition}{Definition}

\newcommand{\statementendmark}{\hfill\textcolor{black!55}{\(\lozenge\)}}
\AtEndEnvironment{theorem}{\statementendmark}
\AtEndEnvironment{maintheorem}{\statementendmark}
\AtEndEnvironment{lemma}{\statementendmark}
\AtEndEnvironment{fact}{\statementendmark}
\AtEndEnvironment{localfact}{\statementendmark}
\AtEndEnvironment{conjecture}{\statementendmark}
\AtEndEnvironment{remark}{\statementendmark}
\AtEndEnvironment{definition}{\statementendmark}
\AtEndEnvironment{localdefinition}{\statementendmark}
\crefname{theorem}{Theorem}{Theorems}
\Crefname{theorem}{Theorem}{Theorems}
\crefname{maintheorem}{Theorem}{Theorems}
\Crefname{maintheorem}{Theorem}{Theorems}
\crefname{lemma}{Lemma}{Lemmas}
\Crefname{lemma}{Lemma}{Lemmas}
\crefname{fact}{Fact}{Facts}
\Crefname{fact}{Fact}{Facts}
\crefname{conjecture}{Conjecture}{Conjectures}
\Crefname{conjecture}{Conjecture}{Conjectures}
\crefname{remark}{Remark}{Remarks}
\Crefname{remark}{Remark}{Remarks}
\crefname{definition}{Definition}{Definitions}
\Crefname{definition}{Definition}{Definitions}

\title{\bfseries A Disproof of Santharoubane's Conjecture\\
on Presentations of Generic Skein Algebras}
\author{Jin-Cheng Guu\\University of Alberta}
\date{}

\providecommand{\buildmetadata}{Build metadata unavailable.}
\fancypagestyle{plain}{%
  \fancyhf{}
  \fancyfoot[L]{\textcolor{gray!25}{\fontsize{6}{7}\selectfont\buildmetadata}}
  \fancyfoot[R]{\textcolor{gray!65}{\scriptsize\thepage}}
  
  }

\begin{document}
\renewcommand{\buildmetadata}{Compiled 2026-08-08 20:53:09 MDT; commit 471ecaaeb985 (clean).}
\maketitle
\vspace{-2.5em}

\begin{abstract}
Let \(\Sigma\) be a compact connected oriented surface of genus at least \(3\)
with at most one boundary component.
Santharoubane associated to certain presentations of the mapping class
group modulo its center a finitely presented algebra equipped with a
canonical surjection onto the generic Kauffman bracket skein algebra of
\(\Sigma\), and conjectured that a suitable choice yields an algebra
isomorphic to the skein algebra.
We show that every algebra arising from this construction admits an
augmentation character, whereas the generic skein algebra of \(\Sigma\)
admits no unital character over \(\mathbb Q(A)\).
The latter obstruction follows from the intersection-one Dehn-twist
identity together with a four-holed-sphere skein relation.
Consequently, the conjectured isomorphism does not hold as stated.
\end{abstract}
\begin{center}
  \footnotesize
  \textbf{Contents}\quad
  \hyperref[sec:introduction]{\ref*{sec:introduction}\ Introduction
  (p.~\pageref*{sec:introduction})}
  \enspace\textperiodcentered\enspace
  \hyperref[sec:preliminaries]{\ref*{sec:preliminaries}\ Preliminaries
  (p.~\pageref*{sec:preliminaries})}
  \enspace\textperiodcentered\enspace
  \hyperref[sec:candidate-algebra]{\ref*{sec:candidate-algebra}\ Candidate
  Algebra (p.~\pageref*{sec:candidate-algebra})}
  \enspace\textperiodcentered\enspace
  \hyperref[sec:main-result]{\ref*{sec:main-result}\ Main Result
  (p.~\pageref*{sec:main-result})}
\end{center}
\vspace{0.25em}
\section{Introduction}\label{sec:introduction}
Kauffman bracket skein algebras form a topological meeting point for
\(\mathrm{SL}_2\)-character varieties and their quantization,
\(U_q(\mathfrak{sl}_2)\) representation theory, \(q\)-special functions,
topological quantum field theory, and string topology.
The skein algebra of a surface is defined topologically from links in its
thickening modulo local relations.
This paper concerns its generic form, in which the skein parameter remains
indeterminate; the integral, generic, and specialized conventions are fixed in
\cref{def:skein-algebra}.
\revdel{At the classical specialization \(A=-1\), the Kauffman bracket
construction recovers the corresponding \(\mathrm{SL}_2\)-character algebra
\parencite{PrzytyckiSikora2000}, while its formal deformation gives a
quantization of the Goldman Poisson structure
\parencite{BullockFrohmanKaniaBartoszynska1999}.}
\revadd{After specializing the integral skein algebra at \(A=-1\) and
extending scalars to \(\mathbb C\), one recovers the coordinate ring of the
\(\mathrm{SL}_2(\mathbb C)\)-character variety
\parencite{PrzytyckiSikora2000}.
The formal skein deformation quantizes the Goldman Poisson bracket
\parencite{BullockFrohmanKaniaBartoszynska1999}.}
The Goldman bracket also enters Chas--Sullivan string topology: for closed
hyperbolic surfaces, Vaintrob describes the string-topology BV operations in
terms of it \parencite{ChasSullivan1999,Vaintrob2007}.
Tham relates skein-theoretic categories to boundary values of the extended
Crane--Yetter TQFT through a skein/state-sum correspondence
\parencite{CraneKauffmanYetter1997,Tham2021}.
Independently, Cooke proves excision for skein categories and identifies the
resulting theory with factorisation homology
\parencite{Cooke2023Excision}.
\par
We briefly survey the relevant presentations and representation-theoretic
models genus by genus.
\begin{itemize}[leftmargin=5em,labelsep=0.75em,itemsep=0.4em,topsep=0.4em,
  before=\interlinepenalty10000\relax]
\item[\(\boldsymbol{g=0}\).]
Cooke and Lacabanne identify skein algebras of punctured spheres (viewed here
as holed spheres) with higher-rank Askey--Wilson algebras; the four-holed
sphere is the rank-one case
\parencite{BullockPrzytycki2000,CookeLacabanne2026}.
The associated Askey--Wilson and \(q\)-Racah polynomials lie at the top of the
\(q\)-Askey scheme, whose specializations and \(q\to1\) limits include familiar
classical families such as Jacobi, Laguerre, and Hermite polynomials
\parencite{AskeyWilson1985,KoekoekLeskySwarttouw2010}.
\revdel{Thus the topology of curves on holed spheres organizes the
Askey--Wilson branch of \(q\)-special-function theory; both skein recoupling
and Askey--Wilson algebras encode \(U_q(\mathfrak{sl}_2)\) tensor-product data
\parencite{CookeLacabanne2026}.}
\revadd{Thus, in genus zero, skein algebras give a topological interpretation
of the Askey--Wilson part of \(q\)-special-function theory
\parencite{CookeLacabanne2026}.}
\item[\(\boldsymbol{g=1}\).]
\revdel{Choosing three slope curves \(y_1,y_2,y_3\), Bullock and Przytycki
present the one-holed-torus skein algebra by the three cyclic relations}
\revadd{Choose curves \(y_1,y_2\) that intersect once, and label \(y_3\)
according to the smoothing convention of
\textcite[Equation~(2.1)]{BullockPrzytycki2000}.
Their Theorem~2.1 presents the one-holed-torus skein algebra by the three
cyclic relations}
\[
  \revdel{[y_i,y_{i+1}]_s=(s^2-s^{-2})y_{i+2}}
  \revadd{[y_i,y_{i+1}]_A=(A^2-A^{-2})y_{i+2}}
  \qquad (i\in\mathbb Z/3\mathbb Z),
\]
and capping the boundary adds a single cubic relation to present the closed
torus \parencite{BullockPrzytycki2000}.
Frohman and Gelca identify the latter algebra with the symmetric part of a
noncommutative torus \parencite{FrohmanGelca2000}.
\revdel{Macdonald theory and double affine Hecke algebras provide the
algebraic background for a second parameter, and related \((q,t)\)-structures
appear in genus-one skein theory and refined Chern--Simons theory
\parencite{EtingofKirillov1994,Cherednik1995,Hikami2019,Samuelson2019,AganagicShakirov2012}.}
\revadd{Macdonald theory and double affine Hecke algebras provide the algebraic
background for a second parameter
\parencite{EtingofKirillov1994,Cherednik1995}.
Related \((q,t)\)-structures appear in genus-one skein theory and refined
Chern--Simons theory
\parencite{Hikami2019,Samuelson2019,AganagicShakirov2012}.}
\item[\(\boldsymbol{g=2}\).]
Arthamonov and Shakirov introduce a two-parameter genus-two analogue of the
spherical \(A_1\) double affine Hecke algebra
\parencite{ArthamonovShakirov2019}.
Cooke and Samuelson show that the closed genus-two skein algebra is generated
by the five Humphries curves \(A_1,A_2,A_3,B_{12},B_{23}\); adjoining the
redundant curve \(B_{13}\) gives a symmetric six-loop calculus
\parencite{CookeSamuelson2021}.
They compute the actions of these loops on the genus-two handlebody skein
module in both theta and dumbbell bases, prove that this module is irreducible,
and describe its restrictions to embedded one-holed-torus and four-holed-sphere
subalgebras using DAHA polynomial representations
\parencite{CookeSamuelson2021}.
\revdel{Cooke and Samuelson further identify the skein algebra, in their
parameter convention, with the diagonal specialization \(q=t=s^4\) of the
genus-two spherical DAHA
\parencite{CookeSamuelson2021,ArthamonovShakirov2019}.}
\revadd{In the convention of Cooke and Samuelson, their skein parameter
\(s\) is our \(A\).
They identify the skein algebra with the diagonal specialization
\(q=t=A^4\) of the genus-two spherical DAHA
\parencite{CookeSamuelson2021,ArthamonovShakirov2019}.}
\revdel{Arthamonov subsequently proves that this family is flat and has the
same monomial basis as its diagonal specialization, making the independent
\(t\)-direction a genuine flat deformation of the genus-two skein algebra
\parencite{Arthamonov2025}.}
\revadd{Arthamonov later proves that this family is flat.
Arthamonov also proves that it has the same monomial basis as its diagonal
specialization.
Thus the independent \(t\)-direction gives a genuine flat deformation of the
genus-two skein algebra \parencite{Arthamonov2025}.}
\revdel{It nevertheless remains a mysterious \((q,t)\)-deformation: away
from the diagonal, no comparably transparent intrinsic topological meaning of
the independent \(t\)-direction is known, and the combinatorial content of
the genus-two DAHA is likewise unclear.}
\revadd{From the skein-theoretic viewpoint, this \((q,t)\)-deformation remains
mysterious.
To the author's knowledge, no transparent intrinsic topological meaning of
the independent \(t\)-direction is known away from the diagonal \(q=t\).}
\item[\(\boldsymbol{g>2}\).]
\revdel{The author knows of no comparably compact finite presentation of the
generic Kauffman bracket skein algebra by globally meaningful curves.}
\revadd{To the author's knowledge, no comparably compact finite presentation
of the generic Kauffman bracket skein algebra by a small family of
topologically specified curves is known.}
A general presentation theorem is nonetheless available: Chen constructs
generators from a cutting graph and proves that the defining ideal is generated
by relations of degree at most six supported on small subsurfaces
\parencite{Chen2024}.
\revdel{This local and effective description does not by itself supply the
small, global topological presentation one would like to compare or deform
across genera.}
\revadd{However, Chen's presentation depends on a cutting graph.
Thus Chen gives a general and effective presentation, but its generators
depend on the chosen cutting graph.}
\revdel{Such a presentation would provide a natural starting point for
higher-genus \((q,t)\)-deformations, whose construction was the original
motivation for the present paper.}
\revadd{The original motivation for this paper was to construct higher-genus
\((q,t)\)-deformations.
A small presentation by topological curves would be a natural starting point.}
\end{itemize}
\par
\revdel{For the \(g>2\) case relevant here, Santharoubane's more general
mapping-class-group construction proposes a route toward such a presentation.}
\revadd{For surfaces of genus \(g>2\), Santharoubane proposed another
candidate presentation based on the mapping class group.}
\revdel{From a finite family of nonseparating simple closed curves that
pairwise intersect at most once and whose Dehn twists generate the mapping
class group, together with a presentation of its quotient by the center,
they construct a finitely presented candidate algebra with a canonical
surjection onto the skein algebra
\parencite[Theorem~1.1, Equation~(4), and
Corollary~1.2]{Santharoubane2024}.}
\revadd{The construction starts with a finite family of nonseparating simple
closed curves.
Any two curves in this family intersect at most once, and their Dehn twists
generate the mapping class group.
One also chooses a presentation of the mapping class group modulo its center.
From these data, Santharoubane constructs a finitely presented candidate
algebra and a canonical surjection onto the skein algebra
\parencite[Theorem~1.1, Equation~(3), and
Corollary~1.2]{Santharoubane2024}.}
\revdel{The conjecture was previously paraphrased as saying that the
canonical surjection is an isomorphism for some presentation datum.}
\revadd{After fixing an admissible curve family,
\textcite[Conjecture~1.3]{Santharoubane2024} asks whether there exists a
presentation of the central quotient for which the associated candidate
algebra is abstractly isomorphic to the skein algebra.}
Unfortunately, in \cref{thm:obstruction} we disprove this conjecture as stated
for surfaces of genus at least \(3\) with at most one boundary component.
\revdel{The obstruction is uniform and elementary: every candidate algebra
has an augmentation character, while the generic skein algebra has no unital
\(\mathbb Q(A)\)-algebra character.}
\revadd{The proof is based on the following observation.
Every candidate algebra has an augmentation character.
On the other hand, the generic skein algebra has no unital
\(\mathbb Q(A)\)-algebra character.}
\par
\section{Preliminaries and notation}\label{sec:preliminaries}

The following notation and conventions remain in force throughout the paper.

\begin{definition}[Coefficient Field and Character]
The symbol \(\mathbb Q\) denotes the field of rational numbers, \(A\) is an
indeterminate, and \(\mathbb Q(A)\) is the field of rational functions in
\(A\) with coefficients in \(\mathbb Q\).
A \emph{character} of a unital \(\mathbb Q(A)\)-algebra \(B\) means a
unital \(\mathbb Q(A)\)-algebra homomorphism \(B\to\mathbb Q(A)\).
\end{definition}

\begin{definition}[Surface and Its Mapping Class Group]
The symbol \(\Sigma\) denotes a compact oriented connected surface of genus
\(g\geq3\) with either no boundary or one boundary component.
A boundary component here is an actual boundary circle of a compact surface,
not a puncture or a removed marked point; no open boundary intervals or marked
boundary points are present.
Thus ``closed'' means \(\partial\Sigma=\varnothing\).
The mapping class group
\[
  \Gamma(\Sigma)=\pi_0\!\left(\operatorname{Homeo}^+(\Sigma,\partial\Sigma)\right)
\]
consists of isotopy classes of orientation-preserving homeomorphisms that fix
\(\partial\Sigma\) pointwise.
The isotopies are also relative to \(\partial\Sigma\): every boundary point
remains fixed throughout the isotopy.
\end{definition}

\begin{definition}[Central Quotient]
For a group \(G\), let \(Z(G)\) denote its center, and set
\[
  \overline{\Gamma(\Sigma)}
  =\Gamma(\Sigma)/Z\bigl(\Gamma(\Sigma)\bigr).
\]
\end{definition}

\begin{definition}[Geometric Intersection and Dehn Twists]
For simple closed curves \(\alpha,\beta\subset\Sigma\), let
\(\iota(\alpha,\beta)\) denote their geometric intersection number.
\revdel{The symbol \(t_\alpha\) denotes the Dehn twist about \(\alpha\).}
\begin{revaddblock}
The symbol \(t_\alpha\) denotes the right Dehn twist about \(\alpha\): in an
orientation-preserving annular chart \(S^1\times[0,1]\) around \(\alpha\), it
is represented by
\[
  t_\alpha(e^{i\theta},r)=(e^{i(\theta-2\pi r)},r),
\]
and it is the identity outside that annulus.
Mapping classes act on skeins by push-forward.
\end{revaddblock}
\end{definition}

\begin{definition}[Kauffman Bracket Skein Algebras and Commutator]
\label{def:skein-algebra}
Put \(\Lambda=\mathbb Z[A,A^{-1}]\), and let
\(\mathcal S_\Lambda(\Sigma)\) be the quotient of the free \(\Lambda\)-module
on isotopy classes of unoriented framed links in the interior of
\(\Sigma\times[0,1]\), including the empty link, by the local Kauffman bracket
relations
\[
  L_\times=A L_\infty+A^{-1}L_0,
  \qquad
  L\sqcup\bigcirc=-(A^2+A^{-2})L.
\]
\revadd{We use the Kauffman-triple convention of
\textcite[Figure~1]{Santharoubane2024} for
\(L_\times,L_\infty,L_0\).}
Here each relation is imposed inside a ball, and \(\bigcirc\) denotes a
zero-framed unknot bounding a disk in that ball.
Multiplication is induced by stacking the left factor above the right factor;
the empty link is the unit.
Via the canonical inclusion \(\Lambda\hookrightarrow\mathbb Q(A)\), the
\emph{generic Kauffman bracket skein algebra} is the scalar extension
\[
  \mathcal S(\Sigma,\mathbb Q(A))
  =\mathbb Q(A)\otimes_\Lambda\mathcal S_\Lambda(\Sigma).
\]
Here \(A\) remains indeterminate.
By contrast, regard \(\mathbb Q\) as a \(\Lambda\)-algebra through the
evaluation homomorphism \(A\mapsto-1\).
The specialization at \(A=-1\) is
\[
  \mathcal S_{-1}(\Sigma)
  =\mathbb Q\otimes_\Lambda\mathcal S_\Lambda(\Sigma)
  \qquad\text{for this \(\Lambda\)-algebra structure on \(\mathbb Q\)}.
\]
Thus the generic algebra studied in this paper and the algebra obtained by
specializing at \(A=-1\) are distinct scalar extensions of the same
Laurent-polynomial skein algebra.
A simple closed curve \(\gamma\subset\Sigma\) is identified with
\(\gamma\times\{1/2\}\), equipped with its blackboard framing.
For \(X,Y\in\mathcal S(\Sigma,\mathbb Q(A))\), define their
\(A\)-commutator by
\begin{equation}\label{eq:q-commutator}
  [X,Y]_A=AXY-A^{-1}YX.
\end{equation}
\end{definition}
\section{The Santharoubane Candidate Algebra}\label{sec:candidate-algebra}
\label{sec:augmentation}

\begin{definition}[Admissible Presentation Datum]
\label{def:admissible-presentation-datum}
An \emph{admissible presentation datum} for Santharoubane's construction
consists of the following data.
First, choose a finite family
\({\protect\hypertarget{mathscope.occurrence.72}{\protect\hyperlink{mathscope.symbol.15fce35597a5eec1}{\{\gamma_i\}_{i\in I}}}\protect\zlabel{mathscope.occurrence.72}}\) satisfying the hypotheses of
\textcite[Theorem~1.1]{Santharoubane2024}: every \({\protect\hypertarget{mathscope.occurrence.73}{\protect\hyperlink{mathscope.symbol.33c46987b405e702}{\gamma_i}}\protect\zlabel{mathscope.occurrence.73}}\) is a
nonseparating simple closed curve, any two curves in the family intersect
at most once, and their Dehn twists generate the mapping class group.
Write \({\protect\hypertarget{mathscope.occurrence.74}{\protect\hyperlink{mathscope.symbol.d7ae2a30537d8b97}{I}}\protect\zlabel{mathscope.occurrence.74}}=\{1,\ldots,{\protect\hypertarget{mathscope.occurrence.75}{\protect\hyperlink{mathscope.symbol.7376975d9cb4a165}{N}}\protect\zlabel{mathscope.occurrence.75}}\}\), set
\({\protect\hypertarget{mathscope.occurrence.76}{\protect\hyperlink{mathscope.symbol.5e984aafe8446c78}{t_i}}\protect\zlabel{mathscope.occurrence.76}}={\protect\hypertarget{mathscope.occurrence.77}{\protect\hyperlink{mathscope.symbol.5e984aafe8446c78}{t_{\gamma_i}}}\protect\zlabel{mathscope.occurrence.77}}\),
\revadd{and use the same symbol \({\protect\hypertarget{mathscope.occurrence.78}{\protect\hyperlink{mathscope.symbol.5e984aafe8446c78}{t_i}}\protect\zlabel{mathscope.occurrence.78}}\) for
the image of \({\protect\hypertarget{mathscope.occurrence.79}{\protect\hyperlink{mathscope.symbol.5e984aafe8446c78}{t_{\gamma_i}}}\protect\zlabel{mathscope.occurrence.79}}\) in
\({\protect\hypertarget{mathscope.occurrence.80}{\protect\hyperlink{mathscope.symbol.98acdb437f3af557}{\overline{\Gamma(\Sigma)}}}\protect\zlabel{mathscope.occurrence.80}}\).}
Second, choose a
presentation, with \({\protect\hypertarget{mathscope.occurrence.81}{\protect\hyperlink{mathscope.symbol.7a6d7bbd508f0a75}{K}}\protect\zlabel{mathscope.occurrence.81}}\geq1\),
\[
  {\protect\hypertarget{mathscope.occurrence.82}{\protect\hyperlink{mathscope.symbol.98acdb437f3af557}{\overline{\Gamma(\Sigma)}}}\protect\zlabel{mathscope.occurrence.82}}
  =
  \left\langle
    {\protect\hypertarget{mathscope.occurrence.83}{\protect\hyperlink{mathscope.symbol.5e984aafe8446c78}{t_1}}\protect\zlabel{mathscope.occurrence.83}},\ldots,{\protect\hypertarget{mathscope.occurrence.84}{\protect\hyperlink{mathscope.symbol.5e984aafe8446c78}{t_N}}\protect\zlabel{mathscope.occurrence.84}}
    \,\middle|\,
    {\protect\hypertarget{mathscope.occurrence.85}{\protect\hyperlink{mathscope.symbol.e3524430777772be}{R_1}}\protect\zlabel{mathscope.occurrence.85}}({\protect\hypertarget{mathscope.occurrence.86}{\protect\hyperlink{mathscope.symbol.5e984aafe8446c78}{t_1}}\protect\zlabel{mathscope.occurrence.86}},\ldots,{\protect\hypertarget{mathscope.occurrence.87}{\protect\hyperlink{mathscope.symbol.5e984aafe8446c78}{t_N}}\protect\zlabel{mathscope.occurrence.87}})=\cdots=
    {\protect\hypertarget{mathscope.occurrence.88}{\protect\hyperlink{mathscope.symbol.e3524430777772be}{R_K}}\protect\zlabel{mathscope.occurrence.88}}({\protect\hypertarget{mathscope.occurrence.89}{\protect\hyperlink{mathscope.symbol.5e984aafe8446c78}{t_1}}\protect\zlabel{mathscope.occurrence.89}},\ldots,{\protect\hypertarget{mathscope.occurrence.90}{\protect\hyperlink{mathscope.symbol.5e984aafe8446c78}{t_N}}\protect\zlabel{mathscope.occurrence.90}})=1
  \right\rangle ,
\]
where each \({\protect\hypertarget{mathscope.occurrence.91}{\protect\hyperlink{mathscope.symbol.e3524430777772be}{R_\ell}}\protect\zlabel{mathscope.occurrence.91}}\) is a word in the letters
\({\protect\hypertarget{mathscope.occurrence.92}{\protect\hyperlink{mathscope.symbol.5e984aafe8446c78}{t_j}}\protect\zlabel{mathscope.occurrence.92}}^{\pm1}\).

We denote this complete admissible presentation datum by
\({\protect\hypertarget{mathscope.occurrence.93}{\protect\hyperlink{mathscope.symbol.92d3b3498b287ed3}{\mathcal P}}\protect\zlabel{mathscope.occurrence.93}}\).
\end{definition}

\begin{definition}[Free Algebra]
Put
\[
  {\protect\hypertarget{mathscope.occurrence.94}{\protect\hyperlink{mathscope.symbol.31f1be226466dba2}{F}}\protect\zlabel{mathscope.occurrence.94}}={\protect\hypertarget{mathscope.occurrence.95}{\protect\hyperlink{mathscope.symbol.793c1bdec545e306}{\mathbb{Q}(A)}}\protect\zlabel{mathscope.occurrence.95}}\langle {\protect\hypertarget{mathscope.occurrence.96}{\protect\hyperlink{mathscope.symbol.2235a5c7430f5467}{X_i}}\protect\zlabel{mathscope.occurrence.96}}:i\in {\protect\hypertarget{mathscope.occurrence.98}{\protect\hyperlink{mathscope.symbol.d7ae2a30537d8b97}{I}}\protect\zlabel{mathscope.occurrence.98}}\rangle
\]
for the free associative, noncommutative \({\protect\hypertarget{mathscope.occurrence.99}{\protect\hyperlink{mathscope.symbol.793c1bdec545e306}{\mathbb{Q}(A)}}\protect\zlabel{mathscope.occurrence.99}}\)-algebra on
the generators \({\protect\hypertarget{mathscope.occurrence.100}{\protect\hyperlink{mathscope.symbol.2235a5c7430f5467}{X_i}}\protect\zlabel{mathscope.occurrence.100}}\).
\end{definition}

\begin{definition}[Augmentation Ideal]
A word of positive length means a nonempty monomial
\({\protect\hypertarget{mathscope.occurrence.101}{\protect\hyperlink{mathscope.symbol.5881a3e9a377a803}{X_{i_1}\cdots X_{i_r}}}\protect\zlabel{mathscope.occurrence.101}}\) with
\(r\geq1\).
Let
\({\protect\hypertarget{mathscope.occurrence.103}{\protect\hyperlink{mathscope.symbol.9846e5ebdbb97960}{\mathfrak m}}\protect\zlabel{mathscope.occurrence.103}}\) be the \({\protect\hypertarget{mathscope.occurrence.104}{\protect\hyperlink{mathscope.symbol.793c1bdec545e306}{\mathbb{Q}(A)}}\protect\zlabel{mathscope.occurrence.104}}\)-span of
all such words.  
Equivalently, \({\protect\hypertarget{mathscope.occurrence.105}{\protect\hyperlink{mathscope.symbol.9846e5ebdbb97960}{\mathfrak m}}\protect\zlabel{mathscope.occurrence.105}}\) is the kernel of the constant-term
augmentation
\begin{equation}\label{eq:free-augmentation}
  {\protect\hypertarget{mathscope.occurrence.106}{\protect\hyperlink{mathscope.symbol.b03c1b8f4b4e4f8d}{\varepsilon_F}}\protect\zlabel{mathscope.occurrence.106}}\colon {\protect\hypertarget{mathscope.occurrence.107}{\protect\hyperlink{mathscope.symbol.31f1be226466dba2}{F}}\protect\zlabel{mathscope.occurrence.107}}\longrightarrow{\protect\hypertarget{mathscope.occurrence.108}{\protect\hyperlink{mathscope.symbol.793c1bdec545e306}{\mathbb{Q}(A)}}\protect\zlabel{mathscope.occurrence.108}},
  \qquad {\protect\hypertarget{mathscope.occurrence.109}{\protect\hyperlink{mathscope.symbol.b03c1b8f4b4e4f8d}{\varepsilon_F}}\protect\zlabel{mathscope.occurrence.109}}({\protect\hypertarget{mathscope.occurrence.110}{\protect\hyperlink{mathscope.symbol.2235a5c7430f5467}{X_i}}\protect\zlabel{mathscope.occurrence.110}})=0.
\end{equation}
\end{definition}

\begin{definition}[Signed Endomorphisms]
For every \(j\in I\) and \(\epsilon\in\{+1,-1\}\),
\revdel{\textcite[Equations~(2)--(3)]{Santharoubane2024} defines algebra
endomorphisms}
\revadd{\textcite[Equation~(2)]{Santharoubane2024} defines algebra
endomorphisms} \({\protect\hypertarget{mathscope.occurrence.112}{\protect\hyperlink{mathscope.symbol.85acf16e72781377}{T_{j,\epsilon}}}\protect\zlabel{mathscope.occurrence.112}}\) of
\({\protect\hypertarget{mathscope.occurrence.113}{\protect\hyperlink{mathscope.symbol.31f1be226466dba2}{F}}\protect\zlabel{mathscope.occurrence.113}}\) by
\begin{equation}\label{eq:santharoubane-action}
  {\protect\hypertarget{mathscope.occurrence.114}{\protect\hyperlink{mathscope.symbol.85acf16e72781377}{T_{j,\epsilon}}}\protect\zlabel{mathscope.occurrence.114}}({\protect\hypertarget{mathscope.occurrence.115}{\protect\hyperlink{mathscope.symbol.2235a5c7430f5467}{X_k}}\protect\zlabel{mathscope.occurrence.115}})=
  \begin{cases}
    {\protect\hypertarget{mathscope.occurrence.116}{\protect\hyperlink{mathscope.symbol.2235a5c7430f5467}{X_k}}\protect\zlabel{mathscope.occurrence.116}},
      & {\protect\hypertarget{mathscope.occurrence.117}{\protect\hyperlink{mathscope.symbol.48fe6ffe25e5deb2}{\iota}}\protect\zlabel{mathscope.occurrence.117}}({\protect\hypertarget{mathscope.occurrence.118}{\protect\hyperlink{mathscope.symbol.33c46987b405e702}{\gamma_j}}\protect\zlabel{mathscope.occurrence.118}},{\protect\hypertarget{mathscope.occurrence.119}{\protect\hyperlink{mathscope.symbol.33c46987b405e702}{\gamma_k}}\protect\zlabel{mathscope.occurrence.119}})=0,\\[2mm]
    \epsilon\bigl({\protect\hypertarget{mathscope.occurrence.121}{\protect\hyperlink{mathscope.symbol.a8c0b3ceba5c59dd}{A}}\protect\zlabel{mathscope.occurrence.121}}^\epsilon {\protect\hypertarget{mathscope.occurrence.123}{\protect\hyperlink{mathscope.symbol.2235a5c7430f5467}{X_j}}\protect\zlabel{mathscope.occurrence.123}}{\protect\hypertarget{mathscope.occurrence.124}{\protect\hyperlink{mathscope.symbol.2235a5c7430f5467}{X_k}}\protect\zlabel{mathscope.occurrence.124}}
      -{\protect\hypertarget{mathscope.occurrence.125}{\protect\hyperlink{mathscope.symbol.a8c0b3ceba5c59dd}{A}}\protect\zlabel{mathscope.occurrence.125}}^{-\epsilon}{\protect\hypertarget{mathscope.occurrence.127}{\protect\hyperlink{mathscope.symbol.2235a5c7430f5467}{X_k}}\protect\zlabel{mathscope.occurrence.127}}{\protect\hypertarget{mathscope.occurrence.128}{\protect\hyperlink{mathscope.symbol.2235a5c7430f5467}{X_j}}\protect\zlabel{mathscope.occurrence.128}}\bigr),
      & {\protect\hypertarget{mathscope.occurrence.129}{\protect\hyperlink{mathscope.symbol.48fe6ffe25e5deb2}{\iota}}\protect\zlabel{mathscope.occurrence.129}}({\protect\hypertarget{mathscope.occurrence.130}{\protect\hyperlink{mathscope.symbol.33c46987b405e702}{\gamma_j}}\protect\zlabel{mathscope.occurrence.130}},{\protect\hypertarget{mathscope.occurrence.131}{\protect\hyperlink{mathscope.symbol.33c46987b405e702}{\gamma_k}}\protect\zlabel{mathscope.occurrence.131}})=1.
  \end{cases}
\end{equation}
Write
\[
  {\protect\hypertarget{mathscope.occurrence.132}{\protect\hyperlink{mathscope.symbol.85acf16e72781377}{T_{j,+}}}\protect\zlabel{mathscope.occurrence.132}}
  :=\left.{\protect\hypertarget{mathscope.occurrence.133}{\protect\hyperlink{mathscope.symbol.85acf16e72781377}{T_{j,\epsilon}}}\protect\zlabel{mathscope.occurrence.133}}\right|_{\epsilon=+1},
  \qquad
  {\protect\hypertarget{mathscope.occurrence.135}{\protect\hyperlink{mathscope.symbol.85acf16e72781377}{T_{j,-}}}\protect\zlabel{mathscope.occurrence.135}}
  :=\left.{\protect\hypertarget{mathscope.occurrence.136}{\protect\hyperlink{mathscope.symbol.85acf16e72781377}{T_{j,\epsilon}}}\protect\zlabel{mathscope.occurrence.136}}\right|_{\epsilon=-1},
  \qquad
  {\protect\hypertarget{mathscope.occurrence.138}{\protect\hyperlink{mathscope.symbol.85acf16e72781377}{T_j}}\protect\zlabel{mathscope.occurrence.138}}:={\protect\hypertarget{mathscope.occurrence.139}{\protect\hyperlink{mathscope.symbol.85acf16e72781377}{T_{j,+}}}\protect\zlabel{mathscope.occurrence.139}}.
\]
\end{definition}

\begin{definition}[Signed Evaluation of Relator Words]
\({\protect\hypertarget{mathscope.occurrence.140}{\protect\hyperlink{mathscope.symbol.e3524430777772be}{R_\ell}}\protect\zlabel{mathscope.occurrence.140}}({\protect\hypertarget{mathscope.occurrence.141}{\protect\hyperlink{mathscope.symbol.85acf16e72781377}{T_1}}\protect\zlabel{mathscope.occurrence.141}},\ldots,
{\protect\hypertarget{mathscope.occurrence.142}{\protect\hyperlink{mathscope.symbol.85acf16e72781377}{T_N}}\protect\zlabel{mathscope.occurrence.142}})\) means signed substitution: each
\({\protect\hypertarget{mathscope.occurrence.143}{\protect\hyperlink{mathscope.symbol.5e984aafe8446c78}{t_j}}\protect\zlabel{mathscope.occurrence.143}}\) is replaced by
\({\protect\hypertarget{mathscope.occurrence.144}{\protect\hyperlink{mathscope.symbol.85acf16e72781377}{T_{j,+}}}\protect\zlabel{mathscope.occurrence.144}}\), and each
\({\protect\hypertarget{mathscope.occurrence.145}{\protect\hyperlink{mathscope.symbol.5e984aafe8446c78}{t_j}}\protect\zlabel{mathscope.occurrence.145}}^{-1}\) is replaced by
\({\protect\hypertarget{mathscope.occurrence.146}{\protect\hyperlink{mathscope.symbol.85acf16e72781377}{T_{j,-}}}\protect\zlabel{mathscope.occurrence.146}}\).
\revadd{The resulting operator word is read as ordinary composition of
endomorphisms, so its rightmost factor acts first.}
\end{definition}

\begin{definition}[Relation Ideal]
For an admissible presentation datum
\({\protect\hypertarget{mathscope.occurrence.147}{\protect\hyperlink{mathscope.symbol.92d3b3498b287ed3}{\mathcal P}}\protect\zlabel{mathscope.occurrence.147}}\), let
\({\protect\hypertarget{mathscope.occurrence.148}{\protect\hyperlink{mathscope.symbol.4bda42526902b34c}{R_{\mathcal P}}}\protect\zlabel{mathscope.occurrence.148}}\subset{\protect\hypertarget{mathscope.occurrence.149}{\protect\hyperlink{mathscope.symbol.31f1be226466dba2}{F}}\protect\zlabel{mathscope.occurrence.149}}\) be
\revdel{the smallest two-sided ideal containing the three families below
and stable under every signed endomorphism.}
\revadd{the ordinary two-sided ideal generated by the following three
families.}
\revadd{Santharoubane calls this a ``bi-ideal'' in the display preceding
Equation~(3); we interpret this as an ordinary two-sided ideal, as confirmed
by the author in private communication
\parencite{SantharoubanePrivate2026}.}

First, it contains the mapping-class action defects
\[
  {\protect\hypertarget{mathscope.occurrence.150}{\protect\hyperlink{mathscope.symbol.e3524430777772be}{R_\ell}}\protect\zlabel{mathscope.occurrence.150}}({\protect\hypertarget{mathscope.occurrence.151}{\protect\hyperlink{mathscope.symbol.85acf16e72781377}{T_1}}\protect\zlabel{mathscope.occurrence.151}},\ldots,
  {\protect\hypertarget{mathscope.occurrence.152}{\protect\hyperlink{mathscope.symbol.85acf16e72781377}{T_N}}\protect\zlabel{mathscope.occurrence.152}})({\protect\hypertarget{mathscope.occurrence.153}{\protect\hyperlink{mathscope.symbol.2235a5c7430f5467}{X_i}}\protect\zlabel{mathscope.occurrence.153}})
  -{\protect\hypertarget{mathscope.occurrence.154}{\protect\hyperlink{mathscope.symbol.2235a5c7430f5467}{X_i}}\protect\zlabel{mathscope.occurrence.154}}
  \qquad
  (1\leq i\leq {\protect\hypertarget{mathscope.occurrence.156}{\protect\hyperlink{mathscope.symbol.7376975d9cb4a165}{N}}\protect\zlabel{mathscope.occurrence.156}},\
   1\leq \ell\leq {\protect\hypertarget{mathscope.occurrence.158}{\protect\hyperlink{mathscope.symbol.7a6d7bbd508f0a75}{K}}\protect\zlabel{mathscope.occurrence.158}}).
\]

Second, it contains the inverse defects
\[
  {\protect\hypertarget{mathscope.occurrence.159}{\protect\hyperlink{mathscope.symbol.85acf16e72781377}{T_{j,+}}}\protect\zlabel{mathscope.occurrence.159}}
  \bigl({\protect\hypertarget{mathscope.occurrence.160}{\protect\hyperlink{mathscope.symbol.85acf16e72781377}{T_{j,-}}}\protect\zlabel{mathscope.occurrence.160}}
  ({\protect\hypertarget{mathscope.occurrence.161}{\protect\hyperlink{mathscope.symbol.2235a5c7430f5467}{X_i}}\protect\zlabel{mathscope.occurrence.161}})\bigr)
  -{\protect\hypertarget{mathscope.occurrence.162}{\protect\hyperlink{mathscope.symbol.2235a5c7430f5467}{X_i}}\protect\zlabel{mathscope.occurrence.162}}
  \qquad
  (1\leq i,j\leq
   {\protect\hypertarget{mathscope.occurrence.165}{\protect\hyperlink{mathscope.symbol.7376975d9cb4a165}{N}}\protect\zlabel{mathscope.occurrence.165}}).
\]

Third, it contains the disjointness commutators
\[
  {\protect\hypertarget{mathscope.occurrence.166}{\protect\hyperlink{mathscope.symbol.2235a5c7430f5467}{X_i}}\protect\zlabel{mathscope.occurrence.166}}{\protect\hypertarget{mathscope.occurrence.167}{\protect\hyperlink{mathscope.symbol.2235a5c7430f5467}{X_j}}\protect\zlabel{mathscope.occurrence.167}}
  -{\protect\hypertarget{mathscope.occurrence.168}{\protect\hyperlink{mathscope.symbol.2235a5c7430f5467}{X_j}}\protect\zlabel{mathscope.occurrence.168}}{\protect\hypertarget{mathscope.occurrence.169}{\protect\hyperlink{mathscope.symbol.2235a5c7430f5467}{X_i}}\protect\zlabel{mathscope.occurrence.169}}
  \qquad
  \bigl({\protect\hypertarget{mathscope.occurrence.170}{\protect\hyperlink{mathscope.symbol.48fe6ffe25e5deb2}{\iota}}\protect\zlabel{mathscope.occurrence.170}}
  ({\protect\hypertarget{mathscope.occurrence.171}{\protect\hyperlink{mathscope.symbol.33c46987b405e702}{\gamma_i}}\protect\zlabel{mathscope.occurrence.171}},{\protect\hypertarget{mathscope.occurrence.172}{\protect\hyperlink{mathscope.symbol.33c46987b405e702}{\gamma_j}}\protect\zlabel{mathscope.occurrence.172}})=0\bigr).
\]
\end{definition}

\begin{definition}[Santharoubane Candidate Algebra]\label{def:candidate-algebra}
\revdel{Following Equation~(4)}\revadd{Following Equation~(3)} of
\textcite{Santharoubane2024}, for an admissible
presentation datum \({\protect\hypertarget{mathscope.occurrence.173}{\protect\hyperlink{mathscope.symbol.92d3b3498b287ed3}{\mathcal P}}\protect\zlabel{mathscope.occurrence.173}}\), define the
\emph{Santharoubane candidate algebra associated to
\({\protect\hypertarget{mathscope.occurrence.174}{\protect\hyperlink{mathscope.symbol.92d3b3498b287ed3}{\mathcal P}}\protect\zlabel{mathscope.occurrence.174}}\)} by
\[
  {\protect\hypertarget{mathscope.occurrence.175}{\protect\hyperlink{mathscope.symbol.a5ff2d70ad4fe714}{\mathcal A(\overline{\Gamma(\Sigma)};\mathcal P)}}\protect\zlabel{mathscope.occurrence.175}}
  :={\protect\hypertarget{mathscope.occurrence.176}{\protect\hyperlink{mathscope.symbol.31f1be226466dba2}{F}}\protect\zlabel{mathscope.occurrence.176}}/{\protect\hypertarget{mathscope.occurrence.177}{\protect\hyperlink{mathscope.symbol.4bda42526902b34c}{R_{\mathcal P}}}\protect\zlabel{mathscope.occurrence.177}}.
\]
\end{definition}

\begin{fact}[Surjection onto Skein Algebra]\label{fact:santharoubane-surjection}
For every admissible presentation datum
\({\protect\hypertarget{mathscope.occurrence.190}{\protect\hyperlink{mathscope.symbol.92d3b3498b287ed3}{\mathcal P}}\protect\zlabel{mathscope.occurrence.190}}\),
\textcite[Corollary~1.2]{Santharoubane2024} gives a canonical surjective
\({\protect\hypertarget{mathscope.occurrence.191}{\protect\hyperlink{mathscope.symbol.793c1bdec545e306}{\mathbb{Q}(A)}}\protect\zlabel{mathscope.occurrence.191}}\)-algebra homomorphism
\[
  {\protect\hypertarget{mathscope.occurrence.192}{\protect\hyperlink{mathscope.symbol.48d22d5de280e7af}{\Psi_{\mathcal P}}}\protect\zlabel{mathscope.occurrence.192}}\colon
  {\protect\hypertarget{mathscope.occurrence.193}{\protect\hyperlink{mathscope.symbol.a5ff2d70ad4fe714}{\mathcal A(\overline{\Gamma(\Sigma)};\mathcal P)}}\protect\zlabel{mathscope.occurrence.193}}
  \longrightarrow
  {\protect\hypertarget{mathscope.occurrence.194}{\protect\hyperlink{mathscope.symbol.2fb732ea957dfec0}{\mathcal{S}(\Sigma,\mathbb{Q}(A))}}\protect\zlabel{mathscope.occurrence.194}}.
\]
\revdel{The generator formula was previously written without its
normalization factor.}
\begin{revaddblock}
By \textcite[Equation~(1)]{Santharoubane2024}, it is determined on
generators by the normalized assignment
\[
  {\protect\hypertarget{mathscope.occurrence.195}{\protect\hyperlink{mathscope.symbol.48d22d5de280e7af}{\Psi_{\mathcal P}}}\protect\zlabel{mathscope.occurrence.195}}\bigl([{\protect\hypertarget{mathscope.occurrence.196}{\protect\hyperlink{mathscope.symbol.2235a5c7430f5467}{X_i}}\protect\zlabel{mathscope.occurrence.196}}]\bigr)
  =\frac{{\protect\hypertarget{mathscope.occurrence.197}{\protect\hyperlink{mathscope.symbol.33c46987b405e702}{\gamma_i}}\protect\zlabel{mathscope.occurrence.197}}}
  {{\protect\hypertarget{mathscope.occurrence.198}{\protect\hyperlink{mathscope.symbol.a8c0b3ceba5c59dd}{A}}\protect\zlabel{mathscope.occurrence.198}}^{2}-{\protect\hypertarget{mathscope.occurrence.199}{\protect\hyperlink{mathscope.symbol.a8c0b3ceba5c59dd}{A}}\protect\zlabel{mathscope.occurrence.199}}^{-2}}.
\]
\end{revaddblock}
\end{fact}

\begin{revaddblock}
\begin{proof}
Since \({\protect\hypertarget{mathscope.occurrence.200}{\protect\hyperlink{mathscope.symbol.31f1be226466dba2}{F}}\protect\zlabel{mathscope.occurrence.200}}\) is free, the normalized assignment
extends uniquely to a \({\protect\hypertarget{mathscope.occurrence.201}{\protect\hyperlink{mathscope.symbol.793c1bdec545e306}{\mathbb{Q}(A)}}\protect\zlabel{mathscope.occurrence.201}}\)-algebra homomorphism
\[
  {\protect\hypertarget{mathscope.occurrence.202}{\protect\hyperlink{mathscope.symbol.dc04d032aad1fbd0}{\widetilde\Psi_{\mathcal P}}}\protect\zlabel{mathscope.occurrence.202}}\colon
  {\protect\hypertarget{mathscope.occurrence.203}{\protect\hyperlink{mathscope.symbol.31f1be226466dba2}{F}}\protect\zlabel{mathscope.occurrence.203}}\longrightarrow
  {\protect\hypertarget{mathscope.occurrence.204}{\protect\hyperlink{mathscope.symbol.2fb732ea957dfec0}{\mathcal{S}(\Sigma,\mathbb{Q}(A))}}\protect\zlabel{mathscope.occurrence.204}},
  \qquad
  {\protect\hypertarget{mathscope.occurrence.205}{\protect\hyperlink{mathscope.symbol.dc04d032aad1fbd0}{\widetilde\Psi_{\mathcal P}}}\protect\zlabel{mathscope.occurrence.205}}
  ({\protect\hypertarget{mathscope.occurrence.206}{\protect\hyperlink{mathscope.symbol.2235a5c7430f5467}{X_i}}\protect\zlabel{mathscope.occurrence.206}})=
  \frac{{\protect\hypertarget{mathscope.occurrence.207}{\protect\hyperlink{mathscope.symbol.33c46987b405e702}{\gamma_i}}\protect\zlabel{mathscope.occurrence.207}}}
  {{\protect\hypertarget{mathscope.occurrence.208}{\protect\hyperlink{mathscope.symbol.a8c0b3ceba5c59dd}{A}}\protect\zlabel{mathscope.occurrence.208}}^{2}-{\protect\hypertarget{mathscope.occurrence.209}{\protect\hyperlink{mathscope.symbol.a8c0b3ceba5c59dd}{A}}\protect\zlabel{mathscope.occurrence.209}}^{-2}}.
\]
Because the relators present the central quotient, each corresponding
mapping class lies in the center of \(\Gamma(\Sigma)\).
The center acts trivially on the skein algebra
\parencite[p.~2573]{Santharoubane2024}.
The intertwining identity in the unnumbered display immediately following
\textcite[Equation~(2)]{Santharoubane2024} therefore sends the mapping-class
defects and inverse defects to zero under
\({\protect\hypertarget{mathscope.occurrence.210}{\protect\hyperlink{mathscope.symbol.dc04d032aad1fbd0}{\widetilde\Psi_{\mathcal P}}}\protect\zlabel{mathscope.occurrence.210}}\); the disjointness commutators are also sent
to zero because disjoint skeins commute.
Thus \({\protect\hypertarget{mathscope.occurrence.211}{\protect\hyperlink{mathscope.symbol.4bda42526902b34c}{R_{\mathcal P}}}\protect\zlabel{mathscope.occurrence.211}}\) is contained in the kernel of
\({\protect\hypertarget{mathscope.occurrence.212}{\protect\hyperlink{mathscope.symbol.dc04d032aad1fbd0}{\widetilde\Psi_{\mathcal P}}}\protect\zlabel{mathscope.occurrence.212}}\), so this map factors through
\({\protect\hypertarget{mathscope.occurrence.213}{\protect\hyperlink{mathscope.symbol.a5ff2d70ad4fe714}{\mathcal A(\overline{\Gamma(\Sigma)};\mathcal P)}}\protect\zlabel{mathscope.occurrence.213}}\) as the displayed map
\({\protect\hypertarget{mathscope.occurrence.214}{\protect\hyperlink{mathscope.symbol.48d22d5de280e7af}{\Psi_{\mathcal P}}}\protect\zlabel{mathscope.occurrence.214}}\).
The curves \({\protect\hypertarget{mathscope.occurrence.215}{\protect\hyperlink{mathscope.symbol.33c46987b405e702}{\gamma_i}}\protect\zlabel{mathscope.occurrence.215}}\) generate
\({\protect\hypertarget{mathscope.occurrence.216}{\protect\hyperlink{mathscope.symbol.2fb732ea957dfec0}{\mathcal{S}(\Sigma,\mathbb{Q}(A))}}\protect\zlabel{mathscope.occurrence.216}}\) by
\textcite[Theorem~1.1]{Santharoubane2024}, and
\({\protect\hypertarget{mathscope.occurrence.217}{\protect\hyperlink{mathscope.symbol.a8c0b3ceba5c59dd}{A}}\protect\zlabel{mathscope.occurrence.217}}^{2}-{\protect\hypertarget{mathscope.occurrence.218}{\protect\hyperlink{mathscope.symbol.a8c0b3ceba5c59dd}{A}}\protect\zlabel{mathscope.occurrence.218}}^{-2}\) is invertible
in \({\protect\hypertarget{mathscope.occurrence.219}{\protect\hyperlink{mathscope.symbol.793c1bdec545e306}{\mathbb{Q}(A)}}\protect\zlabel{mathscope.occurrence.219}}\); hence \({\protect\hypertarget{mathscope.occurrence.220}{\protect\hyperlink{mathscope.symbol.48d22d5de280e7af}{\Psi_{\mathcal P}}}\protect\zlabel{mathscope.occurrence.220}}\) is surjective.
\end{proof}
\end{revaddblock}

\begin{conjecture}[Good Presentation of Skein Algebra]\label{conj:santharoubane}
\revdel{The conjecture was previously stated as requiring the canonical
surjection itself to be an isomorphism.}
\revadd{Fix an arbitrary finite curve family satisfying the curve-family
conditions in \cref{def:admissible-presentation-datum}.
\textcite[Conjecture~1.3]{Santharoubane2024} asks whether there exists a
presentation of
\({\protect\hypertarget{mathscope.occurrence.667}{\protect\hyperlink{mathscope.symbol.98acdb437f3af557}{\overline{\Gamma(\Sigma)}}}\protect\zlabel{mathscope.occurrence.667}}\)
with respect to the corresponding Dehn twists such that, on denoting the
resulting complete datum by
\({\protect\hypertarget{mathscope.occurrence.235}{\protect\hyperlink{mathscope.symbol.92d3b3498b287ed3}{\mathcal P}}\protect\zlabel{mathscope.occurrence.235}}\),
\({\protect\hypertarget{mathscope.occurrence.236}{\protect\hyperlink{mathscope.symbol.a5ff2d70ad4fe714}{\mathcal A(\overline{\Gamma(\Sigma)};\mathcal P)}}\protect\zlabel{mathscope.occurrence.236}}\) is isomorphic to
\({\protect\hypertarget{mathscope.occurrence.237}{\protect\hyperlink{mathscope.symbol.2fb732ea957dfec0}{\mathcal{S}(\Sigma,\mathbb{Q}(A))}}\protect\zlabel{mathscope.occurrence.237}}\) as a noncommutative
\({\protect\hypertarget{mathscope.occurrence.238}{\protect\hyperlink{mathscope.symbol.793c1bdec545e306}{\mathbb{Q}(A)}}\protect\zlabel{mathscope.occurrence.238}}\)-algebra.}
\end{conjecture}

\begin{remark}[Topological content of the conjecture]
\revadd{For each fixed admissible curve family, the conjecture varies only
the chosen presentation of
\({\protect\hypertarget{mathscope.occurrence.240}{\protect\hyperlink{mathscope.symbol.98acdb437f3af557}{\overline{\Gamma(\Sigma)}}}\protect\zlabel{mathscope.occurrence.240}}\).
The notation
\({\protect\hypertarget{mathscope.occurrence.239}{\protect\hyperlink{mathscope.symbol.92d3b3498b287ed3}{\mathcal P}}\protect\zlabel{mathscope.occurrence.239}}\) packages the fixed curve family together
with that presentation.
Theorem~\ref{thm:obstruction} establishes non-isomorphism uniformly over
both choices, which is stronger than needed to refute the cited
conjecture.}
\revdel{If }\revdelnostrike{\cref{conj:santharoubane}}\revdel{ were true, the
displayed quotient would be a finite presentation of the skein algebra
}\revdelnostrike{\({\protect\hypertarget{mathscope.occurrence.241}{\protect\hyperlink{mathscope.symbol.2fb732ea957dfec0}{\mathcal{S}(\Sigma,\mathbb{Q}(A))}}\protect\zlabel{mathscope.occurrence.241}}\)}
\revdel{ with transparent topological meaning.}
\revdel{Its generators map to the unnormalized curve skeins.}
\revadd{For every admissible presentation datum, the canonical surjection
\({\protect\hypertarget{mathscope.occurrence.242}{\protect\hyperlink{mathscope.symbol.48d22d5de280e7af}{\Psi_{\mathcal P}}}\protect\zlabel{mathscope.occurrence.242}}\) sends the displayed generators to the normalized
curve skeins
\({\protect\hypertarget{mathscope.occurrence.243}{\protect\hyperlink{mathscope.symbol.33c46987b405e702}{\gamma_i}}\protect\zlabel{mathscope.occurrence.243}}/
({\protect\hypertarget{mathscope.occurrence.244}{\protect\hyperlink{mathscope.symbol.a8c0b3ceba5c59dd}{A}}\protect\zlabel{mathscope.occurrence.244}}^{2}-{\protect\hypertarget{mathscope.occurrence.245}{\protect\hyperlink{mathscope.symbol.a8c0b3ceba5c59dd}{A}}\protect\zlabel{mathscope.occurrence.245}}^{-2})\), equivalently
to invertible scalar multiples of the specified curves, while its relators arise from a presentation of
\({\protect\hypertarget{mathscope.occurrence.246}{\protect\hyperlink{mathscope.symbol.98acdb437f3af557}{\overline{\Gamma(\Sigma)}}}\protect\zlabel{mathscope.occurrence.246}}\), the Dehn-twist action, and
commutation of disjoint curves.}
\revadd{If this canonical surjection were an isomorphism, it would therefore
give a finite presentation of \({\protect\hypertarget{mathscope.occurrence.247}{\protect\hyperlink{mathscope.symbol.2fb732ea957dfec0}{\mathcal{S}(\Sigma,\mathbb{Q}(A))}}\protect\zlabel{mathscope.occurrence.247}}\) with transparent
topological meaning.}
\revdel{However, }\revdelnostrike{\cref{thm:obstruction}}\revdel{ shows that
this particular presentation is not the generic skein algebra in the range
of the present paper.}
\revadd{Conjecture~\ref{conj:santharoubane} asks only for an abstract
algebra isomorphism, not necessarily for the canonical surjection to be an
isomorphism; \cref{thm:obstruction} rules out even that weaker possibility in
the range of the present paper.}
\end{remark}

\begin{lemma}\label{lem:augmentation}
For every admissible presentation datum
\({\protect\hypertarget{mathscope.occurrence.668}{\protect\hyperlink{mathscope.symbol.92d3b3498b287ed3}{\mathcal P}}\protect\zlabel{mathscope.occurrence.668}}\), the candidate algebra
\({\protect\hypertarget{mathscope.occurrence.248}{\protect\hyperlink{mathscope.symbol.a5ff2d70ad4fe714}{\mathcal A(\overline{\Gamma(\Sigma)};\mathcal P)}}\protect\zlabel{mathscope.occurrence.248}}\) admits a character
\[
  {\protect\hypertarget{mathscope.occurrence.249}{\protect\hyperlink{mathscope.symbol.79ffac822f0cb7af}{\varepsilon_{\mathrm{aug}}}}\protect\zlabel{mathscope.occurrence.249}}\colon{\protect\hypertarget{mathscope.occurrence.250}{\protect\hyperlink{mathscope.symbol.a5ff2d70ad4fe714}{\mathcal A(\overline{\Gamma(\Sigma)};\mathcal P)}}\protect\zlabel{mathscope.occurrence.250}}\longrightarrow{\protect\hypertarget{mathscope.occurrence.251}{\protect\hyperlink{mathscope.symbol.793c1bdec545e306}{\mathbb{Q}(A)}}\protect\zlabel{mathscope.occurrence.251}}
\]
that sends every quotient generator \([{\protect\hypertarget{mathscope.occurrence.252}{\protect\hyperlink{mathscope.symbol.2235a5c7430f5467}{X_i}}\protect\zlabel{mathscope.occurrence.252}}]\) to \(0\).
\end{lemma}

\begin{proof}
Recall from \cref{eq:free-augmentation} that
\({\protect\hypertarget{mathscope.occurrence.253}{\protect\hyperlink{mathscope.symbol.9846e5ebdbb97960}{\mathfrak m}}\protect\zlabel{mathscope.occurrence.253}}\) is the kernel of
\({\protect\hypertarget{mathscope.occurrence.254}{\protect\hyperlink{mathscope.symbol.b03c1b8f4b4e4f8d}{\varepsilon_F}}\protect\zlabel{mathscope.occurrence.254}}\).
For every \(j\) and \(\epsilon\),
\cref{eq:santharoubane-action} gives
\[
  {\protect\hypertarget{mathscope.occurrence.257}{\protect\hyperlink{mathscope.symbol.b03c1b8f4b4e4f8d}{\varepsilon_F}}\protect\zlabel{mathscope.occurrence.257}}\circ
  {\protect\hypertarget{mathscope.occurrence.258}{\protect\hyperlink{mathscope.symbol.85acf16e72781377}{T_{j,\epsilon}}}\protect\zlabel{mathscope.occurrence.258}}
  ={\protect\hypertarget{mathscope.occurrence.259}{\protect\hyperlink{mathscope.symbol.b03c1b8f4b4e4f8d}{\varepsilon_F}}\protect\zlabel{mathscope.occurrence.259}}.
\]
In particular, every
\({\protect\hypertarget{mathscope.occurrence.260}{\protect\hyperlink{mathscope.symbol.85acf16e72781377}{T_{j,\epsilon}}}\protect\zlabel{mathscope.occurrence.260}}\) preserves
\({\protect\hypertarget{mathscope.occurrence.261}{\protect\hyperlink{mathscope.symbol.9846e5ebdbb97960}{\mathfrak m}}\protect\zlabel{mathscope.occurrence.261}}\).
Every displayed generator in the three defining families of
\({\protect\hypertarget{mathscope.occurrence.262}{\protect\hyperlink{mathscope.symbol.4bda42526902b34c}{R_{\mathcal P}}}\protect\zlabel{mathscope.occurrence.262}}\) has zero constant term and therefore
lies in \({\protect\hypertarget{mathscope.occurrence.263}{\protect\hyperlink{mathscope.symbol.9846e5ebdbb97960}{\mathfrak m}}\protect\zlabel{mathscope.occurrence.263}}\).
\revdel{The former operator-stable-closure argument then used this
invariance.}
\revadd{Since \({\protect\hypertarget{mathscope.occurrence.264}{\protect\hyperlink{mathscope.symbol.9846e5ebdbb97960}{\mathfrak m}}\protect\zlabel{mathscope.occurrence.264}}\) is a two-sided ideal, it
contains the ordinary two-sided ideal generated by those three families.}
\revdel{The same conclusion holds if Santharoubane's term
``bi-ideal'' is instead read as the smallest two-sided ideal stable
under all \({\protect\hypertarget{mathscope.occurrence.265}{\protect\hyperlink{mathscope.symbol.85acf16e72781377}{T_{j,\epsilon}}}\protect\zlabel{mathscope.occurrence.265}}\), because
\({\protect\hypertarget{mathscope.occurrence.266}{\protect\hyperlink{mathscope.symbol.9846e5ebdbb97960}{\mathfrak m}}\protect\zlabel{mathscope.occurrence.266}}\) is itself two-sided and stable under
all these endomorphisms.}
Hence \({\protect\hypertarget{mathscope.occurrence.267}{\protect\hyperlink{mathscope.symbol.4bda42526902b34c}{R_{\mathcal P}}}\protect\zlabel{mathscope.occurrence.267}}\subseteq{\protect\hypertarget{mathscope.occurrence.268}{\protect\hyperlink{mathscope.symbol.9846e5ebdbb97960}{\mathfrak m}}\protect\zlabel{mathscope.occurrence.268}}\), and the map
\({\protect\hypertarget{mathscope.occurrence.269}{\protect\hyperlink{mathscope.symbol.b03c1b8f4b4e4f8d}{\varepsilon_F}}\protect\zlabel{mathscope.occurrence.269}}\) from \cref{eq:free-augmentation} descends to the claimed
character.
\end{proof}

\section{Main Result}\label{sec:main-result}
\label{sec:no-character}

The purpose of this section is to prove \cref{thm:obstruction}, which shows
that \cref{conj:santharoubane} is false.
By \cref{lem:augmentation}, the candidate algebra has a character; it
therefore remains to show that the generic skein algebra has none.

\begin{lemma}[Character Obstruction for the Generic Skein Algebra]
\label{lem:no-character}
For every compact oriented connected surface \({\protect\hypertarget{mathscope.occurrence.273}{\protect\hyperlink{mathscope.symbol.e5d086605f4f771b}{\Sigma}}\protect\zlabel{mathscope.occurrence.273}}\) of genus
\({\protect\hypertarget{mathscope.occurrence.274}{\protect\hyperlink{mathscope.symbol.cfb322816069d93a}{g}}\protect\zlabel{mathscope.occurrence.274}}\geq3\) with zero or one boundary component,
\({\protect\hypertarget{mathscope.occurrence.275}{\protect\hyperlink{mathscope.symbol.2fb732ea957dfec0}{\mathcal{S}(\Sigma,\mathbb{Q}(A))}}\protect\zlabel{mathscope.occurrence.275}}\) admits no character.
\end{lemma}

\begin{maintheorem}\label{thm:obstruction}
Let \({\protect\hypertarget{mathscope.occurrence.276}{\protect\hyperlink{mathscope.symbol.e5d086605f4f771b}{\Sigma}}\protect\zlabel{mathscope.occurrence.276}}\) be a compact oriented connected surface of genus \({\protect\hypertarget{mathscope.occurrence.277}{\protect\hyperlink{mathscope.symbol.cfb322816069d93a}{g}}\protect\zlabel{mathscope.occurrence.277}}\geq3\)
with zero or one boundary component.
For every admissible presentation datum
\({\protect\hypertarget{mathscope.occurrence.278}{\protect\hyperlink{mathscope.symbol.92d3b3498b287ed3}{\mathcal P}}\protect\zlabel{mathscope.occurrence.278}}\), the algebra
\({\protect\hypertarget{mathscope.occurrence.279}{\protect\hyperlink{mathscope.symbol.a5ff2d70ad4fe714}{\mathcal A(\overline{\Gamma(\Sigma)};\mathcal P)}}\protect\zlabel{mathscope.occurrence.279}}\) constructed as in
\textcite{Santharoubane2024} is not isomorphic to
\({\protect\hypertarget{mathscope.occurrence.280}{\protect\hyperlink{mathscope.symbol.2fb732ea957dfec0}{\mathcal{S}(\Sigma,\mathbb{Q}(A))}}\protect\zlabel{mathscope.occurrence.280}}\) as a \({\protect\hypertarget{mathscope.occurrence.281}{\protect\hyperlink{mathscope.symbol.793c1bdec545e306}{\mathbb{Q}(A)}}\protect\zlabel{mathscope.occurrence.281}}\)-algebra.
In particular, the canonical surjection \({\protect\hypertarget{mathscope.occurrence.282}{\protect\hyperlink{mathscope.symbol.48d22d5de280e7af}{\Psi_{\mathcal P}}}\protect\zlabel{mathscope.occurrence.282}}\) is never
an isomorphism, so \cref{conj:santharoubane} is false in this range.
\revdel{The same non-isomorphism conclusion holds for the stable-closure
variant of \cref{rem:stable-closure}.}
\end{maintheorem}

\begin{proof}
By \cref{lem:augmentation}, every
\({\protect\hypertarget{mathscope.occurrence.283}{\protect\hyperlink{mathscope.symbol.a5ff2d70ad4fe714}{\mathcal A(\overline{\Gamma(\Sigma)};\mathcal P)}}\protect\zlabel{mathscope.occurrence.283}}\) has a character.
By \cref{lem:no-character}, \({\protect\hypertarget{mathscope.occurrence.284}{\protect\hyperlink{mathscope.symbol.2fb732ea957dfec0}{\mathcal{S}(\Sigma,\mathbb{Q}(A))}}\protect\zlabel{mathscope.occurrence.284}}\) has none.
Since existence of a character is invariant under \({\protect\hypertarget{mathscope.occurrence.285}{\protect\hyperlink{mathscope.symbol.793c1bdec545e306}{\mathbb{Q}(A)}}\protect\zlabel{mathscope.occurrence.285}}\)-algebra isomorphism, the two
algebras cannot be isomorphic: if
\({\protect\hypertarget{mathscope.occurrence.286}{\protect\hyperlink{mathscope.symbol.da542716f1a4c126}{\Phi}}\protect\zlabel{mathscope.occurrence.286}}\colon{\protect\hypertarget{mathscope.occurrence.287}{\protect\hyperlink{mathscope.symbol.a5ff2d70ad4fe714}{\mathcal A(\overline{\Gamma(\Sigma)};\mathcal P)}}\protect\zlabel{mathscope.occurrence.287}}\to{\protect\hypertarget{mathscope.occurrence.288}{\protect\hyperlink{mathscope.symbol.2fb732ea957dfec0}{\mathcal{S}(\Sigma,\mathbb{Q}(A))}}\protect\zlabel{mathscope.occurrence.288}}\) were an
isomorphism and \({\protect\hypertarget{mathscope.occurrence.289}{\protect\hyperlink{mathscope.symbol.dce32381bb291802}{\varepsilon}}\protect\zlabel{mathscope.occurrence.289}}\) were the character from
\cref{lem:augmentation}, then
\({\protect\hypertarget{mathscope.occurrence.290}{\protect\hyperlink{mathscope.symbol.dce32381bb291802}{\varepsilon}}\protect\zlabel{mathscope.occurrence.290}}\circ{\protect\hypertarget{mathscope.occurrence.291}{\protect\hyperlink{mathscope.symbol.da542716f1a4c126}{\Phi}}\protect\zlabel{mathscope.occurrence.291}}^{-1}\) would contradict
\cref{lem:no-character}.
\end{proof}

\revadd{It remains to prove \cref{lem:no-character}.
We prove it in two steps.
First, we show that a hypothetical character either vanishes on every
nonseparating curve or has fixed nonzero magnitude on each such curve, with a
sign that may depend on the curve.
Second, we use one embedded four-holed-sphere relation to rule out both cases.}

\revdel{Proof of the no-character lemma}

\begin{localfact}[Farb--Margalit]\label{fact:farb-margalit-modified-complex}
For \(g\geq2\) and
\(n\geq0\), let
\(S_{g,n}\) be the genus-\(g\)
surface with \(n\) deleted punctures and no
boundary, and define
\({\protect\hypertarget{mathscope.occurrence.297}{\protect\hyperlink{mathscope.symbol.d0eeb1f9202e54ce}{\widehat{\mathcal N}}}\protect\zlabel{mathscope.occurrence.297}}
(S_{g,n})\) to be the one-dimensional simplicial complex
whose vertices are isotopy classes of nonseparating simple closed curves in
\(S_{g,n}\), with an edge between
\([a]\) and \([b]\) exactly when
\({\protect\hypertarget{mathscope.occurrence.302}{\protect\hyperlink{mathscope.symbol.48fe6ffe25e5deb2}{\iota}}\protect\zlabel{mathscope.occurrence.302}}
(a,b)=1\).
Then
\({\protect\hypertarget{mathscope.occurrence.305}{\protect\hyperlink{mathscope.symbol.d0eeb1f9202e54ce}{\widehat{\mathcal N}}}\protect\zlabel{mathscope.occurrence.305}}
(S_{g,n})\) is connected.
\end{localfact}
\begin{proof}
This is \textcite[Lemma~4.5]{FarbMargalit2012}.
\end{proof}

\begin{lemma}[Intersection-One Character Dichotomy]
\label{lem:intersection-one-dichotomy}
Let \({\protect\hypertarget{mathscope.occurrence.307}{\protect\hyperlink{mathscope.symbol.e5d086605f4f771b}{\Sigma}}\protect\zlabel{mathscope.occurrence.307}}\) be a compact oriented connected surface of genus \({\protect\hypertarget{mathscope.occurrence.308}{\protect\hyperlink{mathscope.symbol.cfb322816069d93a}{g}}\protect\zlabel{mathscope.occurrence.308}}\geq3\) with zero or one boundary component, and let \({\protect\hypertarget{mathscope.occurrence.309}{\protect\hyperlink{mathscope.symbol.828377fa8e00251e}{\chi}}\protect\zlabel{mathscope.occurrence.309}}\colon{\protect\hypertarget{mathscope.occurrence.310}{\protect\hyperlink{mathscope.symbol.2fb732ea957dfec0}{\mathcal{S}(\Sigma,\mathbb{Q}(A))}}\protect\zlabel{mathscope.occurrence.310}}\to{\protect\hypertarget{mathscope.occurrence.311}{\protect\hyperlink{mathscope.symbol.793c1bdec545e306}{\mathbb{Q}(A)}}\protect\zlabel{mathscope.occurrence.311}}\) be a character.

For this character, use the following normalized curve values.
For every nonseparating simple closed curve \({\protect\hypertarget{mathscope.occurrence.312}{\protect\hyperlink{mathscope.symbol.cc40d5e6cfccd882}{\alpha}}\protect\zlabel{mathscope.occurrence.312}}\),
set
\begin{equation}\label{eq:normalized-values}
  {\protect\hypertarget{mathscope.occurrence.313}{\protect\hyperlink{mathscope.symbol.a99c17a603551b94}{Y_\alpha}}\protect\zlabel{mathscope.occurrence.313}}=\frac{{\protect\hypertarget{mathscope.occurrence.314}{\protect\hyperlink{mathscope.symbol.cc40d5e6cfccd882}{\alpha}}\protect\zlabel{mathscope.occurrence.314}}}{{\protect\hypertarget{mathscope.occurrence.315}{\protect\hyperlink{mathscope.symbol.a8c0b3ceba5c59dd}{A}}\protect\zlabel{mathscope.occurrence.315}}^2-{\protect\hypertarget{mathscope.occurrence.316}{\protect\hyperlink{mathscope.symbol.a8c0b3ceba5c59dd}{A}}\protect\zlabel{mathscope.occurrence.316}}^{-2}},
  \qquad {\protect\hypertarget{mathscope.occurrence.317}{\protect\hyperlink{mathscope.symbol.a0df08ddf9fe3f71}{y_\alpha}}\protect\zlabel{mathscope.occurrence.317}}={\protect\hypertarget{mathscope.occurrence.318}{\protect\hyperlink{mathscope.symbol.828377fa8e00251e}{\chi}}\protect\zlabel{mathscope.occurrence.318}}({\protect\hypertarget{mathscope.occurrence.319}{\protect\hyperlink{mathscope.symbol.a99c17a603551b94}{Y_\alpha}}\protect\zlabel{mathscope.occurrence.319}}),
  \qquad {\protect\hypertarget{mathscope.occurrence.320}{\protect\hyperlink{mathscope.symbol.a76eff07b1e03bd8}{d}}\protect\zlabel{mathscope.occurrence.320}}={\protect\hypertarget{mathscope.occurrence.321}{\protect\hyperlink{mathscope.symbol.a8c0b3ceba5c59dd}{A}}\protect\zlabel{mathscope.occurrence.321}}-{\protect\hypertarget{mathscope.occurrence.322}{\protect\hyperlink{mathscope.symbol.a8c0b3ceba5c59dd}{A}}\protect\zlabel{mathscope.occurrence.322}}^{-1},
\end{equation}
Then exactly one of the following holds:
\begin{enumerate}[label=(\roman*)]
  \item\label{case:zero} \({\protect\hypertarget{mathscope.occurrence.323}{\protect\hyperlink{mathscope.symbol.a0df08ddf9fe3f71}{y_\alpha}}\protect\zlabel{mathscope.occurrence.323}}=0\) for every nonseparating \({\protect\hypertarget{mathscope.occurrence.324}{\protect\hyperlink{mathscope.symbol.cc40d5e6cfccd882}{\alpha}}\protect\zlabel{mathscope.occurrence.324}}\);
  \item\label{case:nonzero} \({\protect\hypertarget{mathscope.occurrence.325}{\protect\hyperlink{mathscope.symbol.a0df08ddf9fe3f71}{y_\alpha}}\protect\zlabel{mathscope.occurrence.325}}\in\{1,-1\}{\protect\hypertarget{mathscope.occurrence.326}{\protect\hyperlink{mathscope.symbol.a76eff07b1e03bd8}{d}}\protect\zlabel{mathscope.occurrence.326}}^{-1}\) for every nonseparating \({\protect\hypertarget{mathscope.occurrence.327}{\protect\hyperlink{mathscope.symbol.cc40d5e6cfccd882}{\alpha}}\protect\zlabel{mathscope.occurrence.327}}\).
\end{enumerate}
\end{lemma}

\begin{proof}

Let \({\protect\hypertarget{mathscope.occurrence.328}{\protect\hyperlink{mathscope.symbol.cc40d5e6cfccd882}{\alpha}}\protect\zlabel{mathscope.occurrence.328}}\subset{\protect\hypertarget{mathscope.occurrence.329}{\protect\hyperlink{mathscope.symbol.e5d086605f4f771b}{\Sigma}}\protect\zlabel{mathscope.occurrence.329}}\) be an
arbitrary nonseparating simple closed curve.
Cutting \({\protect\hypertarget{mathscope.occurrence.330}{\protect\hyperlink{mathscope.symbol.e5d086605f4f771b}{\Sigma}}\protect\zlabel{mathscope.occurrence.330}}\) along
\({\protect\hypertarget{mathscope.occurrence.331}{\protect\hyperlink{mathscope.symbol.cc40d5e6cfccd882}{\alpha}}\protect\zlabel{mathscope.occurrence.331}}\) leaves a connected surface with two boundary
copies of \({\protect\hypertarget{mathscope.occurrence.332}{\protect\hyperlink{mathscope.symbol.cc40d5e6cfccd882}{\alpha}}\protect\zlabel{mathscope.occurrence.332}}\).
Choose an embedded arc in the cut surface whose endpoints are corresponding
points on the two boundary copies and whose interior is disjoint from the
boundary.
After regluing, this arc closes to a simple closed curve
\({\protect\hypertarget{mathscope.occurrence.333}{\protect\hyperlink{mathscope.symbol.344cc37043ef85d7}{\beta}}\protect\zlabel{mathscope.occurrence.333}}\) satisfying
\({\protect\hypertarget{mathscope.occurrence.334}{\protect\hyperlink{mathscope.symbol.48fe6ffe25e5deb2}{\iota}}\protect\zlabel{mathscope.occurrence.334}}({\protect\hypertarget{mathscope.occurrence.335}{\protect\hyperlink{mathscope.symbol.cc40d5e6cfccd882}{\alpha}}\protect\zlabel{mathscope.occurrence.335}},{\protect\hypertarget{mathscope.occurrence.336}{\protect\hyperlink{mathscope.symbol.344cc37043ef85d7}{\beta}}\protect\zlabel{mathscope.occurrence.336}})=1\).
The curve \({\protect\hypertarget{mathscope.occurrence.337}{\protect\hyperlink{mathscope.symbol.344cc37043ef85d7}{\beta}}\protect\zlabel{mathscope.occurrence.337}}\) is nonseparating, because a separating
curve has zero mod-\(2\) algebraic intersection with every closed curve,
whereas \({\protect\hypertarget{mathscope.occurrence.338}{\protect\hyperlink{mathscope.symbol.344cc37043ef85d7}{\beta}}\protect\zlabel{mathscope.occurrence.338}}\) has mod-\(2\) intersection one with
\({\protect\hypertarget{mathscope.occurrence.339}{\protect\hyperlink{mathscope.symbol.cc40d5e6cfccd882}{\alpha}}\protect\zlabel{mathscope.occurrence.339}}\).
Define \({\protect\hypertarget{mathscope.occurrence.340}{\protect\hyperlink{mathscope.symbol.a99c17a603551b94}{Y_\beta}}\protect\zlabel{mathscope.occurrence.340}}\) by
\cref{eq:normalized-values} with \({\protect\hypertarget{mathscope.occurrence.341}{\protect\hyperlink{mathscope.symbol.344cc37043ef85d7}{\beta}}\protect\zlabel{mathscope.occurrence.341}}\) in place of
\({\protect\hypertarget{mathscope.occurrence.342}{\protect\hyperlink{mathscope.symbol.cc40d5e6cfccd882}{\alpha}}\protect\zlabel{mathscope.occurrence.342}}\).
For this pair, the twist formula from
\revdel{\textcite[Lemma~3.2]{Santharoubane2024} gives}
\revadd{\textcite[Lemma~2.2]{Santharoubane2024} gives}
\begin{equation}\label{eq:intersection-one-twist}
\begin{aligned}
  {\protect\hypertarget{mathscope.occurrence.343}{\protect\hyperlink{mathscope.symbol.5e984aafe8446c78}{t_\alpha}}\protect\zlabel{mathscope.occurrence.343}}({\protect\hypertarget{mathscope.occurrence.344}{\protect\hyperlink{mathscope.symbol.a99c17a603551b94}{Y_\beta}}\protect\zlabel{mathscope.occurrence.344}})
    ={\protect\hypertarget{mathscope.occurrence.345}{\protect\hyperlink{mathscope.symbol.8714a4290c67cc7b}{[Y_\alpha,Y_\beta]_A}}\protect\zlabel{mathscope.occurrence.345}},
  \qquad
  {\protect\hypertarget{mathscope.occurrence.346}{\protect\hyperlink{mathscope.symbol.5e984aafe8446c78}{t_\alpha^{-1}}}\protect\zlabel{mathscope.occurrence.346}}({\protect\hypertarget{mathscope.occurrence.347}{\protect\hyperlink{mathscope.symbol.a99c17a603551b94}{Y_\beta}}\protect\zlabel{mathscope.occurrence.347}})
    ={\protect\hypertarget{mathscope.occurrence.348}{\protect\hyperlink{mathscope.symbol.8714a4290c67cc7b}{[Y_\beta,Y_\alpha]_A}}\protect\zlabel{mathscope.occurrence.348}}.
\end{aligned}
\end{equation}
Since \({\protect\hypertarget{mathscope.occurrence.349}{\protect\hyperlink{mathscope.symbol.5e984aafe8446c78}{t_\alpha}}\protect\zlabel{mathscope.occurrence.349}}\) acts by algebra automorphisms, fixes
\({\protect\hypertarget{mathscope.occurrence.350}{\protect\hyperlink{mathscope.symbol.a99c17a603551b94}{Y_\alpha}}\protect\zlabel{mathscope.occurrence.350}}\), and satisfies
\({\protect\hypertarget{mathscope.occurrence.351}{\protect\hyperlink{mathscope.symbol.5e984aafe8446c78}{t_\alpha}}\protect\zlabel{mathscope.occurrence.351}}{\protect\hypertarget{mathscope.occurrence.352}{\protect\hyperlink{mathscope.symbol.5e984aafe8446c78}{t_\alpha^{-1}}}\protect\zlabel{mathscope.occurrence.352}}
={\protect\hypertarget{mathscope.occurrence.353}{\protect\hyperlink{mathscope.symbol.795ef3503c90622a}{\mathrm{id}}}\protect\zlabel{mathscope.occurrence.353}}\), it preserves the
\({\protect\hypertarget{mathscope.occurrence.354}{\protect\hyperlink{mathscope.symbol.a8c0b3ceba5c59dd}{A}}\protect\zlabel{mathscope.occurrence.354}}\)-commutator:
\({\protect\hypertarget{mathscope.occurrence.355}{\protect\hyperlink{mathscope.symbol.5e984aafe8446c78}{t_\alpha}}\protect\zlabel{mathscope.occurrence.355}}([{\protect\hypertarget{mathscope.occurrence.356}{\protect\hyperlink{mathscope.symbol.c93984e69518acb2}{X}}\protect\zlabel{mathscope.occurrence.356}},
{\protect\hypertarget{mathscope.occurrence.357}{\protect\hyperlink{mathscope.symbol.0e97b6cdc52bf93a}{Y}}\protect\zlabel{mathscope.occurrence.357}}]_{{\protect\hypertarget{mathscope.occurrence.358}{\protect\hyperlink{mathscope.symbol.a8c0b3ceba5c59dd}{A}}\protect\zlabel{mathscope.occurrence.358}}})
=[{\protect\hypertarget{mathscope.occurrence.359}{\protect\hyperlink{mathscope.symbol.5e984aafe8446c78}{t_\alpha}}\protect\zlabel{mathscope.occurrence.359}}({\protect\hypertarget{mathscope.occurrence.360}{\protect\hyperlink{mathscope.symbol.c93984e69518acb2}{X}}\protect\zlabel{mathscope.occurrence.360}}),
{\protect\hypertarget{mathscope.occurrence.361}{\protect\hyperlink{mathscope.symbol.5e984aafe8446c78}{t_\alpha}}\protect\zlabel{mathscope.occurrence.361}}({\protect\hypertarget{mathscope.occurrence.362}{\protect\hyperlink{mathscope.symbol.0e97b6cdc52bf93a}{Y}}\protect\zlabel{mathscope.occurrence.362}})]_
{{\protect\hypertarget{mathscope.occurrence.363}{\protect\hyperlink{mathscope.symbol.a8c0b3ceba5c59dd}{A}}\protect\zlabel{mathscope.occurrence.363}}}\).
Consequently,
\begin{align}
  {\protect\hypertarget{mathscope.occurrence.364}{\protect\hyperlink{mathscope.symbol.a99c17a603551b94}{Y_\beta}}\protect\zlabel{mathscope.occurrence.364}}
  &={\protect\hypertarget{mathscope.occurrence.365}{\protect\hyperlink{mathscope.symbol.5e984aafe8446c78}{t_\alpha}}\protect\zlabel{mathscope.occurrence.365}} {\protect\hypertarget{mathscope.occurrence.366}{\protect\hyperlink{mathscope.symbol.5e984aafe8446c78}{t_\alpha^{-1}}}\protect\zlabel{mathscope.occurrence.366}}({\protect\hypertarget{mathscope.occurrence.367}{\protect\hyperlink{mathscope.symbol.a99c17a603551b94}{Y_\beta}}\protect\zlabel{mathscope.occurrence.367}})\notag\\
  &={\protect\hypertarget{mathscope.occurrence.368}{\protect\hyperlink{mathscope.symbol.5e984aafe8446c78}{t_\alpha}}\protect\zlabel{mathscope.occurrence.368}}\bigl(
    {\protect\hypertarget{mathscope.occurrence.369}{\protect\hyperlink{mathscope.symbol.8714a4290c67cc7b}{[Y_\beta,Y_\alpha]_A}}\protect\zlabel{mathscope.occurrence.369}}\bigr)\notag\\
  &=[{\protect\hypertarget{mathscope.occurrence.370}{\protect\hyperlink{mathscope.symbol.5e984aafe8446c78}{t_\alpha}}\protect\zlabel{mathscope.occurrence.370}}({\protect\hypertarget{mathscope.occurrence.371}{\protect\hyperlink{mathscope.symbol.a99c17a603551b94}{Y_\beta}}\protect\zlabel{mathscope.occurrence.371}}),
      {\protect\hypertarget{mathscope.occurrence.372}{\protect\hyperlink{mathscope.symbol.5e984aafe8446c78}{t_\alpha}}\protect\zlabel{mathscope.occurrence.372}}({\protect\hypertarget{mathscope.occurrence.373}{\protect\hyperlink{mathscope.symbol.a99c17a603551b94}{Y_\alpha}}\protect\zlabel{mathscope.occurrence.373}})]_
      {{\protect\hypertarget{mathscope.occurrence.374}{\protect\hyperlink{mathscope.symbol.a8c0b3ceba5c59dd}{A}}\protect\zlabel{mathscope.occurrence.374}}}\notag\\
  &={\protect\hypertarget{mathscope.occurrence.375}{\protect\hyperlink{mathscope.symbol.8714a4290c67cc7b}{[[Y_\alpha,Y_\beta]_A,Y_\alpha]_A}}\protect\zlabel{mathscope.occurrence.375}}.
\label{eq:nested-commutator}
\end{align}
The four equalities use, respectively, the inverse identity, the second
formula in \cref{eq:intersection-one-twist}, preservation of the
\({\protect\hypertarget{mathscope.occurrence.376}{\protect\hyperlink{mathscope.symbol.a8c0b3ceba5c59dd}{A}}\protect\zlabel{mathscope.occurrence.376}}\)-commutator, and the first formula in
\cref{eq:intersection-one-twist} together with
\({\protect\hypertarget{mathscope.occurrence.377}{\protect\hyperlink{mathscope.symbol.5e984aafe8446c78}{t_\alpha}}\protect\zlabel{mathscope.occurrence.377}}({\protect\hypertarget{mathscope.occurrence.378}{\protect\hyperlink{mathscope.symbol.a99c17a603551b94}{Y_\alpha}}\protect\zlabel{mathscope.occurrence.378}})
={\protect\hypertarget{mathscope.occurrence.379}{\protect\hyperlink{mathscope.symbol.a99c17a603551b94}{Y_\alpha}}\protect\zlabel{mathscope.occurrence.379}}\).
To apply the character \({\protect\hypertarget{mathscope.occurrence.380}{\protect\hyperlink{mathscope.symbol.828377fa8e00251e}{\chi}}\protect\zlabel{mathscope.occurrence.380}}\) to
\cref{eq:nested-commutator}, first observe that, for any
\({\protect\hypertarget{mathscope.occurrence.381}{\protect\hyperlink{mathscope.symbol.c93984e69518acb2}{X}}\protect\zlabel{mathscope.occurrence.381}},{\protect\hypertarget{mathscope.occurrence.382}{\protect\hyperlink{mathscope.symbol.0e97b6cdc52bf93a}{Y}}\protect\zlabel{mathscope.occurrence.382}}
\in{\protect\hypertarget{mathscope.occurrence.383}{\protect\hyperlink{mathscope.symbol.2fb732ea957dfec0}{\mathcal{S}(\Sigma,\mathbb{Q}(A))}}\protect\zlabel{mathscope.occurrence.383}}\), commutativity of the target gives
\begin{align}
  {\protect\hypertarget{mathscope.occurrence.384}{\protect\hyperlink{mathscope.symbol.828377fa8e00251e}{\chi}}\protect\zlabel{mathscope.occurrence.384}}
  \bigl([{\protect\hypertarget{mathscope.occurrence.385}{\protect\hyperlink{mathscope.symbol.c93984e69518acb2}{X}}\protect\zlabel{mathscope.occurrence.385}},{\protect\hypertarget{mathscope.occurrence.386}{\protect\hyperlink{mathscope.symbol.0e97b6cdc52bf93a}{Y}}\protect\zlabel{mathscope.occurrence.386}}]_
    {{\protect\hypertarget{mathscope.occurrence.387}{\protect\hyperlink{mathscope.symbol.a8c0b3ceba5c59dd}{A}}\protect\zlabel{mathscope.occurrence.387}}}\bigr)
  &=
  {\protect\hypertarget{mathscope.occurrence.388}{\protect\hyperlink{mathscope.symbol.828377fa8e00251e}{\chi}}\protect\zlabel{mathscope.occurrence.388}}\bigl(
    {\protect\hypertarget{mathscope.occurrence.389}{\protect\hyperlink{mathscope.symbol.a8c0b3ceba5c59dd}{A}}\protect\zlabel{mathscope.occurrence.389}}{\protect\hypertarget{mathscope.occurrence.390}{\protect\hyperlink{mathscope.symbol.c93984e69518acb2}{X}}\protect\zlabel{mathscope.occurrence.390}}
    {\protect\hypertarget{mathscope.occurrence.391}{\protect\hyperlink{mathscope.symbol.0e97b6cdc52bf93a}{Y}}\protect\zlabel{mathscope.occurrence.391}}
    -{\protect\hypertarget{mathscope.occurrence.392}{\protect\hyperlink{mathscope.symbol.a8c0b3ceba5c59dd}{A}}\protect\zlabel{mathscope.occurrence.392}}^{-1}{\protect\hypertarget{mathscope.occurrence.393}{\protect\hyperlink{mathscope.symbol.0e97b6cdc52bf93a}{Y}}\protect\zlabel{mathscope.occurrence.393}}
    {\protect\hypertarget{mathscope.occurrence.394}{\protect\hyperlink{mathscope.symbol.c93984e69518acb2}{X}}\protect\zlabel{mathscope.occurrence.394}}\bigr)\notag\\
  &=
  {\protect\hypertarget{mathscope.occurrence.395}{\protect\hyperlink{mathscope.symbol.a8c0b3ceba5c59dd}{A}}\protect\zlabel{mathscope.occurrence.395}}
    {\protect\hypertarget{mathscope.occurrence.396}{\protect\hyperlink{mathscope.symbol.828377fa8e00251e}{\chi}}\protect\zlabel{mathscope.occurrence.396}}({\protect\hypertarget{mathscope.occurrence.397}{\protect\hyperlink{mathscope.symbol.c93984e69518acb2}{X}}\protect\zlabel{mathscope.occurrence.397}})
    {\protect\hypertarget{mathscope.occurrence.398}{\protect\hyperlink{mathscope.symbol.828377fa8e00251e}{\chi}}\protect\zlabel{mathscope.occurrence.398}}({\protect\hypertarget{mathscope.occurrence.399}{\protect\hyperlink{mathscope.symbol.0e97b6cdc52bf93a}{Y}}\protect\zlabel{mathscope.occurrence.399}})
  -{\protect\hypertarget{mathscope.occurrence.400}{\protect\hyperlink{mathscope.symbol.a8c0b3ceba5c59dd}{A}}\protect\zlabel{mathscope.occurrence.400}}^{-1}
    {\protect\hypertarget{mathscope.occurrence.401}{\protect\hyperlink{mathscope.symbol.828377fa8e00251e}{\chi}}\protect\zlabel{mathscope.occurrence.401}}({\protect\hypertarget{mathscope.occurrence.402}{\protect\hyperlink{mathscope.symbol.0e97b6cdc52bf93a}{Y}}\protect\zlabel{mathscope.occurrence.402}})
    {\protect\hypertarget{mathscope.occurrence.403}{\protect\hyperlink{mathscope.symbol.828377fa8e00251e}{\chi}}\protect\zlabel{mathscope.occurrence.403}}({\protect\hypertarget{mathscope.occurrence.404}{\protect\hyperlink{mathscope.symbol.c93984e69518acb2}{X}}\protect\zlabel{mathscope.occurrence.404}})\notag\\
  &=
  ({\protect\hypertarget{mathscope.occurrence.405}{\protect\hyperlink{mathscope.symbol.a8c0b3ceba5c59dd}{A}}\protect\zlabel{mathscope.occurrence.405}}-{\protect\hypertarget{mathscope.occurrence.406}{\protect\hyperlink{mathscope.symbol.a8c0b3ceba5c59dd}{A}}\protect\zlabel{mathscope.occurrence.406}}^{-1})
    {\protect\hypertarget{mathscope.occurrence.407}{\protect\hyperlink{mathscope.symbol.828377fa8e00251e}{\chi}}\protect\zlabel{mathscope.occurrence.407}}({\protect\hypertarget{mathscope.occurrence.408}{\protect\hyperlink{mathscope.symbol.c93984e69518acb2}{X}}\protect\zlabel{mathscope.occurrence.408}})
    {\protect\hypertarget{mathscope.occurrence.409}{\protect\hyperlink{mathscope.symbol.828377fa8e00251e}{\chi}}\protect\zlabel{mathscope.occurrence.409}}({\protect\hypertarget{mathscope.occurrence.410}{\protect\hyperlink{mathscope.symbol.0e97b6cdc52bf93a}{Y}}\protect\zlabel{mathscope.occurrence.410}})\notag\\
  &=
  {\protect\hypertarget{mathscope.occurrence.411}{\protect\hyperlink{mathscope.symbol.a76eff07b1e03bd8}{d}}\protect\zlabel{mathscope.occurrence.411}}\,
    {\protect\hypertarget{mathscope.occurrence.412}{\protect\hyperlink{mathscope.symbol.828377fa8e00251e}{\chi}}\protect\zlabel{mathscope.occurrence.412}}({\protect\hypertarget{mathscope.occurrence.413}{\protect\hyperlink{mathscope.symbol.c93984e69518acb2}{X}}\protect\zlabel{mathscope.occurrence.413}})
    {\protect\hypertarget{mathscope.occurrence.414}{\protect\hyperlink{mathscope.symbol.828377fa8e00251e}{\chi}}\protect\zlabel{mathscope.occurrence.414}}({\protect\hypertarget{mathscope.occurrence.415}{\protect\hyperlink{mathscope.symbol.0e97b6cdc52bf93a}{Y}}\protect\zlabel{mathscope.occurrence.415}}).
  \label{eq:character-q-commutator}
\end{align}
Applying \cref{eq:character-q-commutator} first to the outer commutator and
then to the inner commutator yields
\begin{align}
  0
  &={\protect\hypertarget{mathscope.occurrence.416}{\protect\hyperlink{mathscope.symbol.828377fa8e00251e}{\chi}}\protect\zlabel{mathscope.occurrence.416}}\bigl(
      {\protect\hypertarget{mathscope.occurrence.417}{\protect\hyperlink{mathscope.symbol.8714a4290c67cc7b}{[[Y_\alpha,Y_\beta]_A,Y_\alpha]_A}}\protect\zlabel{mathscope.occurrence.417}}
      -{\protect\hypertarget{mathscope.occurrence.418}{\protect\hyperlink{mathscope.symbol.a99c17a603551b94}{Y_\beta}}\protect\zlabel{mathscope.occurrence.418}}\bigr)\notag\\
  &={\protect\hypertarget{mathscope.occurrence.419}{\protect\hyperlink{mathscope.symbol.a76eff07b1e03bd8}{d}}\protect\zlabel{mathscope.occurrence.419}}\,
      {\protect\hypertarget{mathscope.occurrence.420}{\protect\hyperlink{mathscope.symbol.828377fa8e00251e}{\chi}}\protect\zlabel{mathscope.occurrence.420}}
      \bigl({\protect\hypertarget{mathscope.occurrence.421}{\protect\hyperlink{mathscope.symbol.8714a4290c67cc7b}{[Y_\alpha,Y_\beta]_A}}\protect\zlabel{mathscope.occurrence.421}}\bigr)
      {\protect\hypertarget{mathscope.occurrence.422}{\protect\hyperlink{mathscope.symbol.a0df08ddf9fe3f71}{y_\alpha}}\protect\zlabel{mathscope.occurrence.422}}
      -{\protect\hypertarget{mathscope.occurrence.423}{\protect\hyperlink{mathscope.symbol.a0df08ddf9fe3f71}{y_\beta}}\protect\zlabel{mathscope.occurrence.423}}\notag\\
  &={\protect\hypertarget{mathscope.occurrence.424}{\protect\hyperlink{mathscope.symbol.a76eff07b1e03bd8}{d}}\protect\zlabel{mathscope.occurrence.424}}^2
      {\protect\hypertarget{mathscope.occurrence.425}{\protect\hyperlink{mathscope.symbol.a0df08ddf9fe3f71}{y_\alpha}}\protect\zlabel{mathscope.occurrence.425}}^2
      {\protect\hypertarget{mathscope.occurrence.426}{\protect\hyperlink{mathscope.symbol.a0df08ddf9fe3f71}{y_\beta}}\protect\zlabel{mathscope.occurrence.426}}
      -{\protect\hypertarget{mathscope.occurrence.427}{\protect\hyperlink{mathscope.symbol.a0df08ddf9fe3f71}{y_\beta}}\protect\zlabel{mathscope.occurrence.427}}\notag\\
  &={\protect\hypertarget{mathscope.occurrence.428}{\protect\hyperlink{mathscope.symbol.a0df08ddf9fe3f71}{y_\beta}}\protect\zlabel{mathscope.occurrence.428}}
      \bigl({\protect\hypertarget{mathscope.occurrence.429}{\protect\hyperlink{mathscope.symbol.a76eff07b1e03bd8}{d}}\protect\zlabel{mathscope.occurrence.429}}^2
      {\protect\hypertarget{mathscope.occurrence.430}{\protect\hyperlink{mathscope.symbol.a0df08ddf9fe3f71}{y_\alpha}}\protect\zlabel{mathscope.occurrence.430}}^2-1\bigr).
  \label{eq:first-constraint}
\end{align}
Interchanging \({\protect\hypertarget{mathscope.occurrence.431}{\protect\hyperlink{mathscope.symbol.cc40d5e6cfccd882}{\alpha}}\protect\zlabel{mathscope.occurrence.431}}\) and \({\protect\hypertarget{mathscope.occurrence.432}{\protect\hyperlink{mathscope.symbol.344cc37043ef85d7}{\beta}}\protect\zlabel{mathscope.occurrence.432}}\) gives
\begin{equation}\label{eq:second-constraint}
  {\protect\hypertarget{mathscope.occurrence.433}{\protect\hyperlink{mathscope.symbol.a0df08ddf9fe3f71}{y_\alpha}}\protect\zlabel{mathscope.occurrence.433}}({\protect\hypertarget{mathscope.occurrence.434}{\protect\hyperlink{mathscope.symbol.a76eff07b1e03bd8}{d}}\protect\zlabel{mathscope.occurrence.434}}^2{\protect\hypertarget{mathscope.occurrence.435}{\protect\hyperlink{mathscope.symbol.a0df08ddf9fe3f71}{y_\beta}}\protect\zlabel{mathscope.occurrence.435}}^2-1)=0.
\end{equation}
Suppose first that \({\protect\hypertarget{mathscope.occurrence.436}{\protect\hyperlink{mathscope.symbol.a0df08ddf9fe3f71}{y_\alpha}}\protect\zlabel{mathscope.occurrence.436}}=0\).
Then
\cref{eq:first-constraint} reduces to
\(-{\protect\hypertarget{mathscope.occurrence.437}{\protect\hyperlink{mathscope.symbol.a0df08ddf9fe3f71}{y_\beta}}\protect\zlabel{mathscope.occurrence.437}}=0\), so
\({\protect\hypertarget{mathscope.occurrence.438}{\protect\hyperlink{mathscope.symbol.a0df08ddf9fe3f71}{y_\beta}}\protect\zlabel{mathscope.occurrence.438}}=0\).
Conversely, if
\({\protect\hypertarget{mathscope.occurrence.439}{\protect\hyperlink{mathscope.symbol.a0df08ddf9fe3f71}{y_\beta}}\protect\zlabel{mathscope.occurrence.439}}=0\), then
\cref{eq:second-constraint} reduces to
\(-{\protect\hypertarget{mathscope.occurrence.440}{\protect\hyperlink{mathscope.symbol.a0df08ddf9fe3f71}{y_\alpha}}\protect\zlabel{mathscope.occurrence.440}}=0\), so
\({\protect\hypertarget{mathscope.occurrence.441}{\protect\hyperlink{mathscope.symbol.a0df08ddf9fe3f71}{y_\alpha}}\protect\zlabel{mathscope.occurrence.441}}=0\).
If neither value vanishes, divide \cref{eq:first-constraint} by
\({\protect\hypertarget{mathscope.occurrence.442}{\protect\hyperlink{mathscope.symbol.a0df08ddf9fe3f71}{y_\beta}}\protect\zlabel{mathscope.occurrence.442}}\) and
\cref{eq:second-constraint} by \({\protect\hypertarget{mathscope.occurrence.443}{\protect\hyperlink{mathscope.symbol.a0df08ddf9fe3f71}{y_\alpha}}\protect\zlabel{mathscope.occurrence.443}}\) to obtain
\begin{equation}\label{eq:intersection-edge-squares}
  {\protect\hypertarget{mathscope.occurrence.444}{\protect\hyperlink{mathscope.symbol.a76eff07b1e03bd8}{d}}\protect\zlabel{mathscope.occurrence.444}}^2{\protect\hypertarget{mathscope.occurrence.445}{\protect\hyperlink{mathscope.symbol.a0df08ddf9fe3f71}{y_\alpha}}\protect\zlabel{mathscope.occurrence.445}}^2=1,
  \qquad
  {\protect\hypertarget{mathscope.occurrence.446}{\protect\hyperlink{mathscope.symbol.a76eff07b1e03bd8}{d}}\protect\zlabel{mathscope.occurrence.446}}^2{\protect\hypertarget{mathscope.occurrence.447}{\protect\hyperlink{mathscope.symbol.a0df08ddf9fe3f71}{y_\beta}}\protect\zlabel{mathscope.occurrence.447}}^2=1.
\end{equation}
Because \({\protect\hypertarget{mathscope.occurrence.448}{\protect\hyperlink{mathscope.symbol.a76eff07b1e03bd8}{d}}\protect\zlabel{mathscope.occurrence.448}}\neq0\) in \({\protect\hypertarget{mathscope.occurrence.449}{\protect\hyperlink{mathscope.symbol.793c1bdec545e306}{\mathbb{Q}(A)}}\protect\zlabel{mathscope.occurrence.449}}\), this is
equivalent to
\({\protect\hypertarget{mathscope.occurrence.450}{\protect\hyperlink{mathscope.symbol.a0df08ddf9fe3f71}{y_\alpha}}\protect\zlabel{mathscope.occurrence.450}}^2
={\protect\hypertarget{mathscope.occurrence.451}{\protect\hyperlink{mathscope.symbol.a0df08ddf9fe3f71}{y_\beta}}\protect\zlabel{mathscope.occurrence.451}}^2
={\protect\hypertarget{mathscope.occurrence.452}{\protect\hyperlink{mathscope.symbol.a76eff07b1e03bd8}{d}}\protect\zlabel{mathscope.occurrence.452}}^{-2}\).
Thus, whenever the nonseparating curves
\({\protect\hypertarget{mathscope.occurrence.453}{\protect\hyperlink{mathscope.symbol.cc40d5e6cfccd882}{\alpha}}\protect\zlabel{mathscope.occurrence.453}}\) and \({\protect\hypertarget{mathscope.occurrence.454}{\protect\hyperlink{mathscope.symbol.344cc37043ef85d7}{\beta}}\protect\zlabel{mathscope.occurrence.454}}\) satisfy
\({\protect\hypertarget{mathscope.occurrence.455}{\protect\hyperlink{mathscope.symbol.48fe6ffe25e5deb2}{\iota}}\protect\zlabel{mathscope.occurrence.455}}
({\protect\hypertarget{mathscope.occurrence.456}{\protect\hyperlink{mathscope.symbol.cc40d5e6cfccd882}{\alpha}}\protect\zlabel{mathscope.occurrence.456}},{\protect\hypertarget{mathscope.occurrence.457}{\protect\hyperlink{mathscope.symbol.344cc37043ef85d7}{\beta}}\protect\zlabel{mathscope.occurrence.457}})=1\), either
\({\protect\hypertarget{mathscope.occurrence.458}{\protect\hyperlink{mathscope.symbol.a0df08ddf9fe3f71}{y_\alpha}}\protect\zlabel{mathscope.occurrence.458}}={\protect\hypertarget{mathscope.occurrence.459}{\protect\hyperlink{mathscope.symbol.a0df08ddf9fe3f71}{y_\beta}}\protect\zlabel{mathscope.occurrence.459}}=0\), or both
values are nonzero and
\({\protect\hypertarget{mathscope.occurrence.460}{\protect\hyperlink{mathscope.symbol.a0df08ddf9fe3f71}{y_\alpha}}\protect\zlabel{mathscope.occurrence.460}}^2
={\protect\hypertarget{mathscope.occurrence.461}{\protect\hyperlink{mathscope.symbol.a0df08ddf9fe3f71}{y_\beta}}\protect\zlabel{mathscope.occurrence.461}}^2
={\protect\hypertarget{mathscope.occurrence.462}{\protect\hyperlink{mathscope.symbol.a76eff07b1e03bd8}{d}}\protect\zlabel{mathscope.occurrence.462}}^{-2}\).
The construction of \({\protect\hypertarget{mathscope.occurrence.463}{\protect\hyperlink{mathscope.symbol.344cc37043ef85d7}{\beta}}\protect\zlabel{mathscope.occurrence.463}}\) works for every
nonseparating \({\protect\hypertarget{mathscope.occurrence.464}{\protect\hyperlink{mathscope.symbol.cc40d5e6cfccd882}{\alpha}}\protect\zlabel{mathscope.occurrence.464}}\), and the same calculation applies
to every pair of nonseparating curves with geometric intersection number
one.
If \({\protect\hypertarget{mathscope.occurrence.465}{\protect\hyperlink{mathscope.symbol.e5d086605f4f771b}{\Sigma}}\protect\zlabel{mathscope.occurrence.465}}\) is closed, its intersection-one graph of
nonseparating curves is connected by
\textcite[Lemma~4.5]{FarbMargalit2012} with no punctures.
If \({\protect\hypertarget{mathscope.occurrence.466}{\protect\hyperlink{mathscope.symbol.e5d086605f4f771b}{\Sigma}}\protect\zlabel{mathscope.occurrence.466}}\) has one boundary component, its interior
is homeomorphic to the once-punctured closed surface of the same genus, so
the same graph is connected by the one-puncture case of
\textcite[Lemma~4.5]{FarbMargalit2012}.
Connectivity and the preceding intersection-one calculation imply that either
\({\protect\hypertarget{mathscope.occurrence.467}{\protect\hyperlink{mathscope.symbol.a0df08ddf9fe3f71}{y_\alpha}}\protect\zlabel{mathscope.occurrence.467}}=0\) for every nonseparating
\({\protect\hypertarget{mathscope.occurrence.468}{\protect\hyperlink{mathscope.symbol.cc40d5e6cfccd882}{\alpha}}\protect\zlabel{mathscope.occurrence.468}}\), or
\({\protect\hypertarget{mathscope.occurrence.469}{\protect\hyperlink{mathscope.symbol.a0df08ddf9fe3f71}{y_\alpha}}\protect\zlabel{mathscope.occurrence.469}}\neq0\) for every such
\({\protect\hypertarget{mathscope.occurrence.470}{\protect\hyperlink{mathscope.symbol.cc40d5e6cfccd882}{\alpha}}\protect\zlabel{mathscope.occurrence.470}}\).
Either all normalized values vanish, or every normalized value has square
\({\protect\hypertarget{mathscope.occurrence.471}{\protect\hyperlink{mathscope.symbol.a76eff07b1e03bd8}{d}}\protect\zlabel{mathscope.occurrence.471}}^{-2}\), and hence equals
\(\pm{\protect\hypertarget{mathscope.occurrence.472}{\protect\hyperlink{mathscope.symbol.a76eff07b1e03bd8}{d}}\protect\zlabel{mathscope.occurrence.472}}^{-1}\).
The two cases are disjoint.
\end{proof}

\revdel{The Four-Holed-Sphere Obstruction}

To finish the proof of \cref{lem:no-character}, it remains to rule out the
two character patterns in \cref{lem:intersection-one-dichotomy}.
For this it suffices to find a skein relation involving only nonseparating
curves that becomes impossible both when all their character values vanish
and when all their values lie in
\(\{\pm({\protect\hypertarget{mathscope.occurrence.473}{\protect\hyperlink{mathscope.symbol.a8c0b3ceba5c59dd}{A}}\protect\zlabel{mathscope.occurrence.473}}+{\protect\hypertarget{mathscope.occurrence.474}{\protect\hyperlink{mathscope.symbol.a8c0b3ceba5c59dd}{A}}\protect\zlabel{mathscope.occurrence.474}}^{-1})\}\).
The Bullock--Przytycki four-holed-sphere relation has exactly these two
properties: its nonzero scalar term rules out the zero pattern, while its
highest \({\protect\hypertarget{mathscope.occurrence.475}{\protect\hyperlink{mathscope.symbol.a8c0b3ceba5c59dd}{A}}\protect\zlabel{mathscope.occurrence.475}}\)-degree rules out the second pattern.
\begin{localdefinition}[Ordered Bullock--Przytycki configuration]
Let \({\protect\hypertarget{mathscope.occurrence.476}{\protect\hyperlink{mathscope.symbol.fa5dff9603ca46cc}{P}}\protect\zlabel{mathscope.occurrence.476}}\cong{\protect\hypertarget{mathscope.occurrence.477}{\protect\hyperlink{mathscope.symbol.938578720bae9414}{\Sigma_{0,4}}}\protect\zlabel{mathscope.occurrence.477}}\) be an oriented
four-holed sphere.
For \(1\leq i\leq 4\), let
\({\protect\hypertarget{mathscope.occurrence.479}{\protect\hyperlink{mathscope.symbol.9826f0776fa9144e}{a_i}}\protect\zlabel{mathscope.occurrence.479}}\subset\operatorname{int}({\protect\hypertarget{mathscope.occurrence.480}{\protect\hyperlink{mathscope.symbol.fa5dff9603ca46cc}{P}}\protect\zlabel{mathscope.occurrence.480}})\)
be a boundary-parallel push-off of the
\(i\)-th boundary component.
Let
\(({\protect\hypertarget{mathscope.occurrence.482}{\protect\hyperlink{mathscope.symbol.7eaae3bb6234079a}{z_1}}\protect\zlabel{mathscope.occurrence.482}},{\protect\hypertarget{mathscope.occurrence.483}{\protect\hyperlink{mathscope.symbol.7eaae3bb6234079a}{z_2}}\protect\zlabel{mathscope.occurrence.483}},{\protect\hypertarget{mathscope.occurrence.484}{\protect\hyperlink{mathscope.symbol.7eaae3bb6234079a}{z_3}}\protect\zlabel{mathscope.occurrence.484}})\)
\begin{revaddblock}
be a choice of essential curves separating, respectively,
\[
  {\protect\hypertarget{mathscope.occurrence.485}{\protect\hyperlink{mathscope.symbol.7eaae3bb6234079a}{z_1}}\protect\zlabel{mathscope.occurrence.485}}\colon(12\mid34),\qquad
  {\protect\hypertarget{mathscope.occurrence.486}{\protect\hyperlink{mathscope.symbol.7eaae3bb6234079a}{z_2}}\protect\zlabel{mathscope.occurrence.486}}\colon(23\mid14),\qquad
  {\protect\hypertarget{mathscope.occurrence.487}{\protect\hyperlink{mathscope.symbol.7eaae3bb6234079a}{z_3}}\protect\zlabel{mathscope.occurrence.487}}\colon(13\mid24).
\]
The product convention is fixed by requiring
\({\protect\hypertarget{mathscope.occurrence.491}{\protect\hyperlink{mathscope.symbol.7eaae3bb6234079a}{z_3}}\protect\zlabel{mathscope.occurrence.491}}\) to be the essential curve occurring with
coefficient \({\protect\hypertarget{mathscope.occurrence.492}{\protect\hyperlink{mathscope.symbol.a8c0b3ceba5c59dd}{A}}\protect\zlabel{mathscope.occurrence.492}}^{2}\) in the resolution of the
ordered product
\({\protect\hypertarget{mathscope.occurrence.493}{\protect\hyperlink{mathscope.symbol.7eaae3bb6234079a}{z_1}}\protect\zlabel{mathscope.occurrence.493}}{\protect\hypertarget{mathscope.occurrence.494}{\protect\hyperlink{mathscope.symbol.7eaae3bb6234079a}{z_2}}\protect\zlabel{mathscope.occurrence.494}}\) in
\textcite[Equation~(3.1)]{BullockPrzytycki2000}; the boundary labels agree
with \textcite[Figure~3]{CookeSamuelson2021}.
\[
  {\protect\hypertarget{mathscope.occurrence.495}{\protect\hyperlink{mathscope.symbol.48fe6ffe25e5deb2}{\iota}}\protect\zlabel{mathscope.occurrence.495}}
    ({\protect\hypertarget{mathscope.occurrence.496}{\protect\hyperlink{mathscope.symbol.7eaae3bb6234079a}{z_1}}\protect\zlabel{mathscope.occurrence.496}},{\protect\hypertarget{mathscope.occurrence.497}{\protect\hyperlink{mathscope.symbol.7eaae3bb6234079a}{z_2}}\protect\zlabel{mathscope.occurrence.497}})
  ={\protect\hypertarget{mathscope.occurrence.498}{\protect\hyperlink{mathscope.symbol.48fe6ffe25e5deb2}{\iota}}\protect\zlabel{mathscope.occurrence.498}}
    ({\protect\hypertarget{mathscope.occurrence.499}{\protect\hyperlink{mathscope.symbol.7eaae3bb6234079a}{z_2}}\protect\zlabel{mathscope.occurrence.499}},{\protect\hypertarget{mathscope.occurrence.500}{\protect\hyperlink{mathscope.symbol.7eaae3bb6234079a}{z_3}}\protect\zlabel{mathscope.occurrence.500}})
  ={\protect\hypertarget{mathscope.occurrence.501}{\protect\hyperlink{mathscope.symbol.48fe6ffe25e5deb2}{\iota}}\protect\zlabel{mathscope.occurrence.501}}
    ({\protect\hypertarget{mathscope.occurrence.502}{\protect\hyperlink{mathscope.symbol.7eaae3bb6234079a}{z_3}}\protect\zlabel{mathscope.occurrence.502}},{\protect\hypertarget{mathscope.occurrence.503}{\protect\hyperlink{mathscope.symbol.7eaae3bb6234079a}{z_1}}\protect\zlabel{mathscope.occurrence.503}})
  =2.
\]
Thus every pair in the ordered triple intersects twice.
Any orientation-preserving, boundary-label-preserving image of this ordered
triple may be used, since the relation below is natural under surface
homeomorphisms.
\end{revaddblock}
\end{localdefinition}

\begin{localfact}[Four-Holed-Sphere Relation]
After base change from \({\protect\hypertarget{mathscope.occurrence.504}{\protect\hyperlink{mathscope.symbol.940d6398c67789c9}{\mathbb{Z}[A^{\pm1}]}}\protect\zlabel{mathscope.occurrence.504}}\) to \({\protect\hypertarget{mathscope.occurrence.505}{\protect\hyperlink{mathscope.symbol.793c1bdec545e306}{\mathbb{Q}(A)}}\protect\zlabel{mathscope.occurrence.505}}\), the
four-holed-sphere relation in
\textcite[Equation~(3.2)]{BullockPrzytycki2000} is
\begin{align}
  {\protect\hypertarget{mathscope.occurrence.506}{\protect\hyperlink{mathscope.symbol.a8c0b3ceba5c59dd}{A}}\protect\zlabel{mathscope.occurrence.506}}^2{\protect\hypertarget{mathscope.occurrence.507}{\protect\hyperlink{mathscope.symbol.7eaae3bb6234079a}{z_1}}\protect\zlabel{mathscope.occurrence.507}}{\protect\hypertarget{mathscope.occurrence.508}{\protect\hyperlink{mathscope.symbol.7eaae3bb6234079a}{z_2}}\protect\zlabel{mathscope.occurrence.508}}{\protect\hypertarget{mathscope.occurrence.509}{\protect\hyperlink{mathscope.symbol.7eaae3bb6234079a}{z_3}}\protect\zlabel{mathscope.occurrence.509}}
  ={}&{\protect\hypertarget{mathscope.occurrence.510}{\protect\hyperlink{mathscope.symbol.a8c0b3ceba5c59dd}{A}}\protect\zlabel{mathscope.occurrence.510}}^4{\protect\hypertarget{mathscope.occurrence.511}{\protect\hyperlink{mathscope.symbol.7eaae3bb6234079a}{z_1}}\protect\zlabel{mathscope.occurrence.511}}^2+{\protect\hypertarget{mathscope.occurrence.512}{\protect\hyperlink{mathscope.symbol.a8c0b3ceba5c59dd}{A}}\protect\zlabel{mathscope.occurrence.512}}^{-4}{\protect\hypertarget{mathscope.occurrence.513}{\protect\hyperlink{mathscope.symbol.7eaae3bb6234079a}{z_2}}\protect\zlabel{mathscope.occurrence.513}}^2+{\protect\hypertarget{mathscope.occurrence.514}{\protect\hyperlink{mathscope.symbol.a8c0b3ceba5c59dd}{A}}\protect\zlabel{mathscope.occurrence.514}}^4{\protect\hypertarget{mathscope.occurrence.515}{\protect\hyperlink{mathscope.symbol.7eaae3bb6234079a}{z_3}}\protect\zlabel{mathscope.occurrence.515}}^2 \notag\\
    &+{\protect\hypertarget{mathscope.occurrence.516}{\protect\hyperlink{mathscope.symbol.a8c0b3ceba5c59dd}{A}}\protect\zlabel{mathscope.occurrence.516}}^2{\protect\hypertarget{mathscope.occurrence.517}{\protect\hyperlink{mathscope.symbol.8b496f65f4dd0c3b}{p_1}}\protect\zlabel{mathscope.occurrence.517}}{\protect\hypertarget{mathscope.occurrence.518}{\protect\hyperlink{mathscope.symbol.7eaae3bb6234079a}{z_1}}\protect\zlabel{mathscope.occurrence.518}}+{\protect\hypertarget{mathscope.occurrence.519}{\protect\hyperlink{mathscope.symbol.a8c0b3ceba5c59dd}{A}}\protect\zlabel{mathscope.occurrence.519}}^{-2}{\protect\hypertarget{mathscope.occurrence.520}{\protect\hyperlink{mathscope.symbol.8b496f65f4dd0c3b}{p_2}}\protect\zlabel{mathscope.occurrence.520}}{\protect\hypertarget{mathscope.occurrence.521}{\protect\hyperlink{mathscope.symbol.7eaae3bb6234079a}{z_2}}\protect\zlabel{mathscope.occurrence.521}}+{\protect\hypertarget{mathscope.occurrence.522}{\protect\hyperlink{mathscope.symbol.a8c0b3ceba5c59dd}{A}}\protect\zlabel{mathscope.occurrence.522}}^2{\protect\hypertarget{mathscope.occurrence.523}{\protect\hyperlink{mathscope.symbol.8b496f65f4dd0c3b}{p_3}}\protect\zlabel{mathscope.occurrence.523}}{\protect\hypertarget{mathscope.occurrence.524}{\protect\hyperlink{mathscope.symbol.7eaae3bb6234079a}{z_3}}\protect\zlabel{mathscope.occurrence.524}} \notag\\
    &+{\protect\hypertarget{mathscope.occurrence.525}{\protect\hyperlink{mathscope.symbol.21dc1d4c093d8950}{q}}\protect\zlabel{mathscope.occurrence.525}}-({\protect\hypertarget{mathscope.occurrence.526}{\protect\hyperlink{mathscope.symbol.a8c0b3ceba5c59dd}{A}}\protect\zlabel{mathscope.occurrence.526}}^2+{\protect\hypertarget{mathscope.occurrence.527}{\protect\hyperlink{mathscope.symbol.a8c0b3ceba5c59dd}{A}}\protect\zlabel{mathscope.occurrence.527}}^{-2})^2, \label{eq:four-holed-relation}
\end{align}
where
\begin{align}
  {\protect\hypertarget{mathscope.occurrence.528}{\protect\hyperlink{mathscope.symbol.8b496f65f4dd0c3b}{p_1}}\protect\zlabel{mathscope.occurrence.528}}&={\protect\hypertarget{mathscope.occurrence.529}{\protect\hyperlink{mathscope.symbol.9826f0776fa9144e}{a_1}}\protect\zlabel{mathscope.occurrence.529}}{\protect\hypertarget{mathscope.occurrence.530}{\protect\hyperlink{mathscope.symbol.9826f0776fa9144e}{a_2}}\protect\zlabel{mathscope.occurrence.530}}+{\protect\hypertarget{mathscope.occurrence.531}{\protect\hyperlink{mathscope.symbol.9826f0776fa9144e}{a_3}}\protect\zlabel{mathscope.occurrence.531}}{\protect\hypertarget{mathscope.occurrence.532}{\protect\hyperlink{mathscope.symbol.9826f0776fa9144e}{a_4}}\protect\zlabel{mathscope.occurrence.532}},&
  {\protect\hypertarget{mathscope.occurrence.533}{\protect\hyperlink{mathscope.symbol.8b496f65f4dd0c3b}{p_2}}\protect\zlabel{mathscope.occurrence.533}}&={\protect\hypertarget{mathscope.occurrence.534}{\protect\hyperlink{mathscope.symbol.9826f0776fa9144e}{a_2}}\protect\zlabel{mathscope.occurrence.534}}{\protect\hypertarget{mathscope.occurrence.535}{\protect\hyperlink{mathscope.symbol.9826f0776fa9144e}{a_3}}\protect\zlabel{mathscope.occurrence.535}}+{\protect\hypertarget{mathscope.occurrence.536}{\protect\hyperlink{mathscope.symbol.9826f0776fa9144e}{a_1}}\protect\zlabel{mathscope.occurrence.536}}{\protect\hypertarget{mathscope.occurrence.537}{\protect\hyperlink{mathscope.symbol.9826f0776fa9144e}{a_4}}\protect\zlabel{mathscope.occurrence.537}},&
  {\protect\hypertarget{mathscope.occurrence.538}{\protect\hyperlink{mathscope.symbol.8b496f65f4dd0c3b}{p_3}}\protect\zlabel{mathscope.occurrence.538}}&={\protect\hypertarget{mathscope.occurrence.539}{\protect\hyperlink{mathscope.symbol.9826f0776fa9144e}{a_1}}\protect\zlabel{mathscope.occurrence.539}}{\protect\hypertarget{mathscope.occurrence.540}{\protect\hyperlink{mathscope.symbol.9826f0776fa9144e}{a_3}}\protect\zlabel{mathscope.occurrence.540}}+{\protect\hypertarget{mathscope.occurrence.541}{\protect\hyperlink{mathscope.symbol.9826f0776fa9144e}{a_2}}\protect\zlabel{mathscope.occurrence.541}}{\protect\hypertarget{mathscope.occurrence.542}{\protect\hyperlink{mathscope.symbol.9826f0776fa9144e}{a_4}}\protect\zlabel{mathscope.occurrence.542}}, \label{eq:p-definitions}\\
  {\protect\hypertarget{mathscope.occurrence.543}{\protect\hyperlink{mathscope.symbol.21dc1d4c093d8950}{q}}\protect\zlabel{mathscope.occurrence.543}}&={\protect\hypertarget{mathscope.occurrence.544}{\protect\hyperlink{mathscope.symbol.9826f0776fa9144e}{a_1}}\protect\zlabel{mathscope.occurrence.544}}{\protect\hypertarget{mathscope.occurrence.545}{\protect\hyperlink{mathscope.symbol.9826f0776fa9144e}{a_2}}\protect\zlabel{mathscope.occurrence.545}}{\protect\hypertarget{mathscope.occurrence.546}{\protect\hyperlink{mathscope.symbol.9826f0776fa9144e}{a_3}}\protect\zlabel{mathscope.occurrence.546}}{\protect\hypertarget{mathscope.occurrence.547}{\protect\hyperlink{mathscope.symbol.9826f0776fa9144e}{a_4}}\protect\zlabel{mathscope.occurrence.547}}+{\protect\hypertarget{mathscope.occurrence.548}{\protect\hyperlink{mathscope.symbol.9826f0776fa9144e}{a_1}}\protect\zlabel{mathscope.occurrence.548}}^2+{\protect\hypertarget{mathscope.occurrence.549}{\protect\hyperlink{mathscope.symbol.9826f0776fa9144e}{a_2}}\protect\zlabel{mathscope.occurrence.549}}^2+{\protect\hypertarget{mathscope.occurrence.550}{\protect\hyperlink{mathscope.symbol.9826f0776fa9144e}{a_3}}\protect\zlabel{mathscope.occurrence.550}}^2+{\protect\hypertarget{mathscope.occurrence.551}{\protect\hyperlink{mathscope.symbol.9826f0776fa9144e}{a_4}}\protect\zlabel{mathscope.occurrence.551}}^2.
  \label{eq:q-definition}
\end{align}
\end{localfact}
\begin{proof}
This is \textcite[Equation~(3.2)]{BullockPrzytycki2000}, after base change
to \({\protect\hypertarget{mathscope.occurrence.552}{\protect\hyperlink{mathscope.symbol.793c1bdec545e306}{\mathbb{Q}(A)}}\protect\zlabel{mathscope.occurrence.552}}\); its skein and trivial-loop conventions agree
with \cref{def:skein-algebra}, and the labels agree with
\textcite[Figure~3 and Theorem~2.16]{CookeSamuelson2021}.
\end{proof}

\revdel{It remains only to prove the no-character lemma; together with the
augmentation lemma, this proves the main theorem and disproves the
conjecture.}

\begin{proof}[Proof of \cref{lem:no-character}]
Suppose that a character
\({\protect\hypertarget{mathscope.occurrence.553}{\protect\hyperlink{mathscope.symbol.828377fa8e00251e}{\chi}}\protect\zlabel{mathscope.occurrence.553}}\colon
{\protect\hypertarget{mathscope.occurrence.554}{\protect\hyperlink{mathscope.symbol.2fb732ea957dfec0}{\mathcal{S}(\Sigma,\mathbb{Q}(A))}}\protect\zlabel{mathscope.occurrence.554}}\to{\protect\hypertarget{mathscope.occurrence.555}{\protect\hyperlink{mathscope.symbol.793c1bdec545e306}{\mathbb{Q}(A)}}\protect\zlabel{mathscope.occurrence.555}}\) exists.
\revdel{The previous wording said that the normalized dichotomy directly
gave the corresponding alternatives for the raw curve skeins.}
\begin{revaddblock}
Multiplying the normalized alternatives in
\cref{lem:intersection-one-dichotomy} by
\({\protect\hypertarget{mathscope.occurrence.556}{\protect\hyperlink{mathscope.symbol.a8c0b3ceba5c59dd}{A}}\protect\zlabel{mathscope.occurrence.556}}^{2}-{\protect\hypertarget{mathscope.occurrence.557}{\protect\hyperlink{mathscope.symbol.a8c0b3ceba5c59dd}{A}}\protect\zlabel{mathscope.occurrence.557}}^{-2}\), and using
\[
  \frac{{\protect\hypertarget{mathscope.occurrence.558}{\protect\hyperlink{mathscope.symbol.a8c0b3ceba5c59dd}{A}}\protect\zlabel{mathscope.occurrence.558}}^{2}
  -{\protect\hypertarget{mathscope.occurrence.559}{\protect\hyperlink{mathscope.symbol.a8c0b3ceba5c59dd}{A}}\protect\zlabel{mathscope.occurrence.559}}^{-2}}
  {{\protect\hypertarget{mathscope.occurrence.560}{\protect\hyperlink{mathscope.symbol.a8c0b3ceba5c59dd}{A}}\protect\zlabel{mathscope.occurrence.560}}-{\protect\hypertarget{mathscope.occurrence.561}{\protect\hyperlink{mathscope.symbol.a8c0b3ceba5c59dd}{A}}\protect\zlabel{mathscope.occurrence.561}}^{-1}}
  ={\protect\hypertarget{mathscope.occurrence.562}{\protect\hyperlink{mathscope.symbol.a8c0b3ceba5c59dd}{A}}\protect\zlabel{mathscope.occurrence.562}}+{\protect\hypertarget{mathscope.occurrence.563}{\protect\hyperlink{mathscope.symbol.a8c0b3ceba5c59dd}{A}}\protect\zlabel{mathscope.occurrence.563}}^{-1},
\]
shows that either
\end{revaddblock}
\({\protect\hypertarget{mathscope.occurrence.564}{\protect\hyperlink{mathscope.symbol.828377fa8e00251e}{\chi}}\protect\zlabel{mathscope.occurrence.564}}({\protect\hypertarget{mathscope.occurrence.565}{\protect\hyperlink{mathscope.symbol.cc40d5e6cfccd882}{\alpha}}\protect\zlabel{mathscope.occurrence.565}})=0\) for every
nonseparating \({\protect\hypertarget{mathscope.occurrence.566}{\protect\hyperlink{mathscope.symbol.cc40d5e6cfccd882}{\alpha}}\protect\zlabel{mathscope.occurrence.566}}\), or every such value is
\(\pm({\protect\hypertarget{mathscope.occurrence.567}{\protect\hyperlink{mathscope.symbol.a8c0b3ceba5c59dd}{A}}\protect\zlabel{mathscope.occurrence.567}}+{\protect\hypertarget{mathscope.occurrence.568}{\protect\hyperlink{mathscope.symbol.a8c0b3ceba5c59dd}{A}}\protect\zlabel{mathscope.occurrence.568}}^{-1})\).
Choose an embedding
\({\protect\hypertarget{mathscope.occurrence.569}{\protect\hyperlink{mathscope.symbol.fa5dff9603ca46cc}{P}}\protect\zlabel{mathscope.occurrence.569}}\hookrightarrow{\protect\hypertarget{mathscope.occurrence.570}{\protect\hyperlink{mathscope.symbol.e5d086605f4f771b}{\Sigma}}\protect\zlabel{mathscope.occurrence.570}}\)
with connected complement.
\revadd{Such an orientation-preserving embedding exists by gluing the four
boundary circles of \({\protect\hypertarget{mathscope.occurrence.571}{\protect\hyperlink{mathscope.symbol.fa5dff9603ca46cc}{P}}\protect\zlabel{mathscope.occurrence.571}}\) to four boundary circles
of a connected genus-\(({\protect\hypertarget{mathscope.occurrence.572}{\protect\hyperlink{mathscope.symbol.cfb322816069d93a}{g}}\protect\zlabel{mathscope.occurrence.572}}-3)\) surface by
orientation-reversing boundary identifications, leaving
exactly the boundary circle of \({\protect\hypertarget{mathscope.occurrence.573}{\protect\hyperlink{mathscope.symbol.e5d086605f4f771b}{\Sigma}}\protect\zlabel{mathscope.occurrence.573}}\), if any,
unglued.}
\revadd{After cutting along any
\({\protect\hypertarget{mathscope.occurrence.574}{\protect\hyperlink{mathscope.symbol.9826f0776fa9144e}{a_i}}\protect\zlabel{mathscope.occurrence.574}}\), the two sides remain joined through the other
three gluing circles; after cutting along any
\({\protect\hypertarget{mathscope.occurrence.575}{\protect\hyperlink{mathscope.symbol.7eaae3bb6234079a}{z_i}}\protect\zlabel{mathscope.occurrence.575}}\), the two resulting pairs of pants remain joined
through the connected complement.  Hence all four
\({\protect\hypertarget{mathscope.occurrence.576}{\protect\hyperlink{mathscope.symbol.9826f0776fa9144e}{a_i}}\protect\zlabel{mathscope.occurrence.576}}\) and all three
\({\protect\hypertarget{mathscope.occurrence.577}{\protect\hyperlink{mathscope.symbol.7eaae3bb6234079a}{z_i}}\protect\zlabel{mathscope.occurrence.577}}\) are nonseparating in
\({\protect\hypertarget{mathscope.occurrence.578}{\protect\hyperlink{mathscope.symbol.e5d086605f4f771b}{\Sigma}}\protect\zlabel{mathscope.occurrence.578}}\).}
By \textcite[Proposition~1.1(3)]{BullockPrzytycki2000}, the induced
skein-algebra homomorphism makes \cref{eq:four-holed-relation} an identity
in \({\protect\hypertarget{mathscope.occurrence.579}{\protect\hyperlink{mathscope.symbol.2fb732ea957dfec0}{\mathcal{S}(\Sigma,\mathbb{Q}(A))}}\protect\zlabel{mathscope.occurrence.579}}\).

If \({\protect\hypertarget{mathscope.occurrence.580}{\protect\hyperlink{mathscope.symbol.828377fa8e00251e}{\chi}}\protect\zlabel{mathscope.occurrence.580}}\) vanishes on every nonseparating simple closed curve, applying it to
\cref{eq:four-holed-relation} gives
\[
  0=-({\protect\hypertarget{mathscope.occurrence.581}{\protect\hyperlink{mathscope.symbol.a8c0b3ceba5c59dd}{A}}\protect\zlabel{mathscope.occurrence.581}}^2+{\protect\hypertarget{mathscope.occurrence.582}{\protect\hyperlink{mathscope.symbol.a8c0b3ceba5c59dd}{A}}\protect\zlabel{mathscope.occurrence.582}}^{-2})^2,
\]
which is impossible in \({\protect\hypertarget{mathscope.occurrence.583}{\protect\hyperlink{mathscope.symbol.793c1bdec545e306}{\mathbb{Q}(A)}}\protect\zlabel{mathscope.occurrence.583}}\).

In the second branch, put
\({\protect\hypertarget{mathscope.occurrence.584}{\protect\hyperlink{mathscope.symbol.146522e6ef302a05}{U}}\protect\zlabel{mathscope.occurrence.584}}={\protect\hypertarget{mathscope.occurrence.585}{\protect\hyperlink{mathscope.symbol.a8c0b3ceba5c59dd}{A}}\protect\zlabel{mathscope.occurrence.585}}
+{\protect\hypertarget{mathscope.occurrence.586}{\protect\hyperlink{mathscope.symbol.a8c0b3ceba5c59dd}{A}}\protect\zlabel{mathscope.occurrence.586}}^{-1}\), and write
\begin{align*}
  {\protect\hypertarget{mathscope.occurrence.587}{\protect\hyperlink{mathscope.symbol.828377fa8e00251e}{\chi}}\protect\zlabel{mathscope.occurrence.587}}({\protect\hypertarget{mathscope.occurrence.588}{\protect\hyperlink{mathscope.symbol.9826f0776fa9144e}{a_1}}\protect\zlabel{mathscope.occurrence.588}})
    &={\protect\hypertarget{mathscope.occurrence.589}{\protect\hyperlink{mathscope.symbol.3c539076f6aef6ce}{\varepsilon_1}}\protect\zlabel{mathscope.occurrence.589}}{\protect\hypertarget{mathscope.occurrence.590}{\protect\hyperlink{mathscope.symbol.146522e6ef302a05}{U}}\protect\zlabel{mathscope.occurrence.590}},&
  {\protect\hypertarget{mathscope.occurrence.591}{\protect\hyperlink{mathscope.symbol.828377fa8e00251e}{\chi}}\protect\zlabel{mathscope.occurrence.591}}({\protect\hypertarget{mathscope.occurrence.592}{\protect\hyperlink{mathscope.symbol.9826f0776fa9144e}{a_2}}\protect\zlabel{mathscope.occurrence.592}})
    &={\protect\hypertarget{mathscope.occurrence.593}{\protect\hyperlink{mathscope.symbol.3c539076f6aef6ce}{\varepsilon_2}}\protect\zlabel{mathscope.occurrence.593}}{\protect\hypertarget{mathscope.occurrence.594}{\protect\hyperlink{mathscope.symbol.146522e6ef302a05}{U}}\protect\zlabel{mathscope.occurrence.594}},&
  {\protect\hypertarget{mathscope.occurrence.595}{\protect\hyperlink{mathscope.symbol.828377fa8e00251e}{\chi}}\protect\zlabel{mathscope.occurrence.595}}({\protect\hypertarget{mathscope.occurrence.596}{\protect\hyperlink{mathscope.symbol.9826f0776fa9144e}{a_3}}\protect\zlabel{mathscope.occurrence.596}})
    &={\protect\hypertarget{mathscope.occurrence.597}{\protect\hyperlink{mathscope.symbol.3c539076f6aef6ce}{\varepsilon_3}}\protect\zlabel{mathscope.occurrence.597}}{\protect\hypertarget{mathscope.occurrence.598}{\protect\hyperlink{mathscope.symbol.146522e6ef302a05}{U}}\protect\zlabel{mathscope.occurrence.598}},&
  {\protect\hypertarget{mathscope.occurrence.599}{\protect\hyperlink{mathscope.symbol.828377fa8e00251e}{\chi}}\protect\zlabel{mathscope.occurrence.599}}({\protect\hypertarget{mathscope.occurrence.600}{\protect\hyperlink{mathscope.symbol.9826f0776fa9144e}{a_4}}\protect\zlabel{mathscope.occurrence.600}})
    &={\protect\hypertarget{mathscope.occurrence.601}{\protect\hyperlink{mathscope.symbol.3c539076f6aef6ce}{\varepsilon_4}}\protect\zlabel{mathscope.occurrence.601}}{\protect\hypertarget{mathscope.occurrence.602}{\protect\hyperlink{mathscope.symbol.146522e6ef302a05}{U}}\protect\zlabel{mathscope.occurrence.602}},\\
  {\protect\hypertarget{mathscope.occurrence.603}{\protect\hyperlink{mathscope.symbol.828377fa8e00251e}{\chi}}\protect\zlabel{mathscope.occurrence.603}}({\protect\hypertarget{mathscope.occurrence.604}{\protect\hyperlink{mathscope.symbol.7eaae3bb6234079a}{z_1}}\protect\zlabel{mathscope.occurrence.604}})
    &={\protect\hypertarget{mathscope.occurrence.605}{\protect\hyperlink{mathscope.symbol.1d23208b547b148d}{\delta_1}}\protect\zlabel{mathscope.occurrence.605}}{\protect\hypertarget{mathscope.occurrence.606}{\protect\hyperlink{mathscope.symbol.146522e6ef302a05}{U}}\protect\zlabel{mathscope.occurrence.606}},&
  {\protect\hypertarget{mathscope.occurrence.607}{\protect\hyperlink{mathscope.symbol.828377fa8e00251e}{\chi}}\protect\zlabel{mathscope.occurrence.607}}({\protect\hypertarget{mathscope.occurrence.608}{\protect\hyperlink{mathscope.symbol.7eaae3bb6234079a}{z_2}}\protect\zlabel{mathscope.occurrence.608}})
    &={\protect\hypertarget{mathscope.occurrence.609}{\protect\hyperlink{mathscope.symbol.1d23208b547b148d}{\delta_2}}\protect\zlabel{mathscope.occurrence.609}}{\protect\hypertarget{mathscope.occurrence.610}{\protect\hyperlink{mathscope.symbol.146522e6ef302a05}{U}}\protect\zlabel{mathscope.occurrence.610}},&
  {\protect\hypertarget{mathscope.occurrence.611}{\protect\hyperlink{mathscope.symbol.828377fa8e00251e}{\chi}}\protect\zlabel{mathscope.occurrence.611}}({\protect\hypertarget{mathscope.occurrence.612}{\protect\hyperlink{mathscope.symbol.7eaae3bb6234079a}{z_3}}\protect\zlabel{mathscope.occurrence.612}})
    &={\protect\hypertarget{mathscope.occurrence.613}{\protect\hyperlink{mathscope.symbol.1d23208b547b148d}{\delta_3}}\protect\zlabel{mathscope.occurrence.613}}{\protect\hypertarget{mathscope.occurrence.614}{\protect\hyperlink{mathscope.symbol.146522e6ef302a05}{U}}\protect\zlabel{mathscope.occurrence.614}},&&
\end{align*}
where all seven signs lie in \(\{\pm1\}\).
Set
\begin{align*}
  B
  &=
  {\protect\hypertarget{mathscope.occurrence.616}{\protect\hyperlink{mathscope.symbol.1d23208b547b148d}{\delta_1}}\protect\zlabel{mathscope.occurrence.616}}
    ({\protect\hypertarget{mathscope.occurrence.617}{\protect\hyperlink{mathscope.symbol.3c539076f6aef6ce}{\varepsilon_1}}\protect\zlabel{mathscope.occurrence.617}}{\protect\hypertarget{mathscope.occurrence.618}{\protect\hyperlink{mathscope.symbol.3c539076f6aef6ce}{\varepsilon_2}}\protect\zlabel{mathscope.occurrence.618}}
     +{\protect\hypertarget{mathscope.occurrence.619}{\protect\hyperlink{mathscope.symbol.3c539076f6aef6ce}{\varepsilon_3}}\protect\zlabel{mathscope.occurrence.619}}{\protect\hypertarget{mathscope.occurrence.620}{\protect\hyperlink{mathscope.symbol.3c539076f6aef6ce}{\varepsilon_4}}\protect\zlabel{mathscope.occurrence.620}})
  +{\protect\hypertarget{mathscope.occurrence.621}{\protect\hyperlink{mathscope.symbol.1d23208b547b148d}{\delta_3}}\protect\zlabel{mathscope.occurrence.621}}
    ({\protect\hypertarget{mathscope.occurrence.622}{\protect\hyperlink{mathscope.symbol.3c539076f6aef6ce}{\varepsilon_1}}\protect\zlabel{mathscope.occurrence.622}}{\protect\hypertarget{mathscope.occurrence.623}{\protect\hyperlink{mathscope.symbol.3c539076f6aef6ce}{\varepsilon_3}}\protect\zlabel{mathscope.occurrence.623}}
     +{\protect\hypertarget{mathscope.occurrence.624}{\protect\hyperlink{mathscope.symbol.3c539076f6aef6ce}{\varepsilon_2}}\protect\zlabel{mathscope.occurrence.624}}{\protect\hypertarget{mathscope.occurrence.625}{\protect\hyperlink{mathscope.symbol.3c539076f6aef6ce}{\varepsilon_4}}\protect\zlabel{mathscope.occurrence.625}})
  -{\protect\hypertarget{mathscope.occurrence.626}{\protect\hyperlink{mathscope.symbol.1d23208b547b148d}{\delta_1}}\protect\zlabel{mathscope.occurrence.626}}{\protect\hypertarget{mathscope.occurrence.627}{\protect\hyperlink{mathscope.symbol.1d23208b547b148d}{\delta_2}}\protect\zlabel{mathscope.occurrence.627}}{\protect\hypertarget{mathscope.occurrence.628}{\protect\hyperlink{mathscope.symbol.1d23208b547b148d}{\delta_3}}\protect\zlabel{mathscope.occurrence.628}},\\
  C
  &=
  {\protect\hypertarget{mathscope.occurrence.630}{\protect\hyperlink{mathscope.symbol.1d23208b547b148d}{\delta_2}}\protect\zlabel{mathscope.occurrence.630}}
    ({\protect\hypertarget{mathscope.occurrence.631}{\protect\hyperlink{mathscope.symbol.3c539076f6aef6ce}{\varepsilon_2}}\protect\zlabel{mathscope.occurrence.631}}{\protect\hypertarget{mathscope.occurrence.632}{\protect\hyperlink{mathscope.symbol.3c539076f6aef6ce}{\varepsilon_3}}\protect\zlabel{mathscope.occurrence.632}}
     +{\protect\hypertarget{mathscope.occurrence.633}{\protect\hyperlink{mathscope.symbol.3c539076f6aef6ce}{\varepsilon_1}}\protect\zlabel{mathscope.occurrence.633}}{\protect\hypertarget{mathscope.occurrence.634}{\protect\hyperlink{mathscope.symbol.3c539076f6aef6ce}{\varepsilon_4}}\protect\zlabel{mathscope.occurrence.634}}),\\
  E
  &={\protect\hypertarget{mathscope.occurrence.636}{\protect\hyperlink{mathscope.symbol.3c539076f6aef6ce}{\varepsilon_1}}\protect\zlabel{mathscope.occurrence.636}}{\protect\hypertarget{mathscope.occurrence.637}{\protect\hyperlink{mathscope.symbol.3c539076f6aef6ce}{\varepsilon_2}}\protect\zlabel{mathscope.occurrence.637}}
    {\protect\hypertarget{mathscope.occurrence.638}{\protect\hyperlink{mathscope.symbol.3c539076f6aef6ce}{\varepsilon_3}}\protect\zlabel{mathscope.occurrence.638}}{\protect\hypertarget{mathscope.occurrence.639}{\protect\hyperlink{mathscope.symbol.3c539076f6aef6ce}{\varepsilon_4}}\protect\zlabel{mathscope.occurrence.639}}.
\end{align*}
Applying \({\protect\hypertarget{mathscope.occurrence.640}{\protect\hyperlink{mathscope.symbol.828377fa8e00251e}{\chi}}\protect\zlabel{mathscope.occurrence.640}}\) to
\cref{eq:four-holed-relation} and expanding gives
\begin{align*}
0={}&2{\protect\hypertarget{mathscope.occurrence.641}{\protect\hyperlink{mathscope.symbol.a8c0b3ceba5c59dd}{A}}\protect\zlabel{mathscope.occurrence.641}}^{6}
+B{\protect\hypertarget{mathscope.occurrence.643}{\protect\hyperlink{mathscope.symbol.a8c0b3ceba5c59dd}{A}}\protect\zlabel{mathscope.occurrence.643}}^{5}
+(3+E){\protect\hypertarget{mathscope.occurrence.645}{\protect\hyperlink{mathscope.symbol.a8c0b3ceba5c59dd}{A}}\protect\zlabel{mathscope.occurrence.645}}^{4}
+3B{\protect\hypertarget{mathscope.occurrence.647}{\protect\hyperlink{mathscope.symbol.a8c0b3ceba5c59dd}{A}}\protect\zlabel{mathscope.occurrence.647}}^{3}\\
&+(6+4E){\protect\hypertarget{mathscope.occurrence.649}{\protect\hyperlink{mathscope.symbol.a8c0b3ceba5c59dd}{A}}\protect\zlabel{mathscope.occurrence.649}}^{2}
+(3B+C)
  {\protect\hypertarget{mathscope.occurrence.652}{\protect\hyperlink{mathscope.symbol.a8c0b3ceba5c59dd}{A}}\protect\zlabel{mathscope.occurrence.652}}
+6+6E\\
&+(B+3C)
  {\protect\hypertarget{mathscope.occurrence.656}{\protect\hyperlink{mathscope.symbol.a8c0b3ceba5c59dd}{A}}\protect\zlabel{mathscope.occurrence.656}}^{-1}
+(5+4E){\protect\hypertarget{mathscope.occurrence.658}{\protect\hyperlink{mathscope.symbol.a8c0b3ceba5c59dd}{A}}\protect\zlabel{mathscope.occurrence.658}}^{-2}
+3C{\protect\hypertarget{mathscope.occurrence.660}{\protect\hyperlink{mathscope.symbol.a8c0b3ceba5c59dd}{A}}\protect\zlabel{mathscope.occurrence.660}}^{-3}\\
&+(1+E){\protect\hypertarget{mathscope.occurrence.662}{\protect\hyperlink{mathscope.symbol.a8c0b3ceba5c59dd}{A}}\protect\zlabel{mathscope.occurrence.662}}^{-4}
+C{\protect\hypertarget{mathscope.occurrence.664}{\protect\hyperlink{mathscope.symbol.a8c0b3ceba5c59dd}{A}}\protect\zlabel{mathscope.occurrence.664}}^{-5}
+{\protect\hypertarget{mathscope.occurrence.665}{\protect\hyperlink{mathscope.symbol.a8c0b3ceba5c59dd}{A}}\protect\zlabel{mathscope.occurrence.665}}^{-6}.
\end{align*}
Its \({\protect\hypertarget{mathscope.occurrence.666}{\protect\hyperlink{mathscope.symbol.a8c0b3ceba5c59dd}{A}}\protect\zlabel{mathscope.occurrence.666}}^{6}\)-coefficient is \(2\), a contradiction.

\end{proof}
\section*{Acknowledgments}

I thank Ramanujan Santharoubane, Terry Gannon, and Harshit Yadav for helpful
discussions. I also thank OpenAI's Codex for assistance with literature
organization, proofreading and editing, and reproducibility checks. I
gratefully acknowledge that this research was supported in part by
the Pacific Institute for the Mathematical Sciences. This work was supported by
a grant from the Simons Foundation International
[SFI-MPS-T-Institutes-00020822, OY].
\hypersetup{urlcolor=InternalLink}
\printbibliography

@article{BullockPrzytycki2000,
  author       = {Bullock, Doug and Przytycki, J{\'o}zef H.},
  title        = {Multiplicative Structure of {Kauffman} Bracket Skein
                  Module Quantizations},
  journaltitle = {Proceedings of the American Mathematical Society},
  volume       = {128},
  number       = {3},
  year         = {2000},
  pages        = {923--931},
  doi          = {10.1090/S0002-9939-99-05043-1},
  eprint       = {math/9902117},
  eprinttype   = {arXiv},
  url          = {https://pubs.ams.org/journals/proc/2000-128-03/S0002-9939-99-05043-1}
}

@article{PrzytyckiSikora2000,
  author       = {Przytycki, J{\'o}zef H. and Sikora, Adam S.},
  title        = {On Skein Algebras and \(SL_2(\mathbb C)\)-Character
                  Varieties},
  journaltitle = {Topology},
  volume       = {39},
  number       = {1},
  year         = {2000},
  pages        = {115--148},
  doi          = {10.1016/S0040-9383(98)00062-7},
  eprint       = {q-alg/9705011},
  eprinttype   = {arXiv},
  url          = {https://www.sciencedirect.com/science/article/pii/S0040938398000627}
}

@article{BullockFrohmanKaniaBartoszynska1999,
  author       = {Bullock, Doug and Frohman, Charles and
                  Kania-Bartoszy{\'n}ska, Joanna},
  title        = {Understanding the {Kauffman} Bracket Skein Module},
  journaltitle = {Journal of Knot Theory and Its Ramifications},
  volume       = {8},
  number       = {3},
  year         = {1999},
  pages        = {265--277},
  doi          = {10.1142/S0218216599000183},
  eprint       = {q-alg/9604013},
  eprinttype   = {arXiv},
  url          = {https://www.worldscientific.com/doi/10.1142/S0218216599000183}
}

@online{ChasSullivan1999,
  author     = {Chas, Moira and Sullivan, Dennis},
  title      = {String Topology},
  year       = {1999},
  eprint     = {math/9911159},
  eprinttype = {arXiv},
  url        = {https://arxiv.org/abs/math/9911159}
}

@online{Vaintrob2007,
  author     = {Vaintrob, Dmitry},
  title      = {The String Topology {BV} Algebra, Hochschild Cohomology and
                the Goldman Bracket on Surfaces},
  year       = {2007},
  eprint     = {math/0702859},
  eprinttype = {arXiv},
  url        = {https://arxiv.org/abs/math/0702859}
}

@article{FrohmanGelca2000,
  author       = {Frohman, Charles and Gelca, Razvan},
  title        = {Skein Modules and the Noncommutative Torus},
  journaltitle = {Transactions of the American Mathematical Society},
  volume       = {352},
  number       = {10},
  year         = {2000},
  pages        = {4877--4888},
  doi          = {10.1090/S0002-9947-00-02512-5},
  eprint       = {math/9806107},
  eprinttype   = {arXiv},
  url          = {https://pubs.ams.org/journals/tran/2000-352-10/S0002-9947-00-02512-5}
}

@article{Cooke2023Excision,
  author       = {Cooke, Juliet},
  title        = {Excision of Skein Categories and Factorisation Homology},
  journaltitle = {Advances in Mathematics},
  volume       = {414},
  year         = {2023},
  eid          = {108848},
  doi          = {10.1016/j.aim.2022.108848},
  eprint       = {1910.02630},
  eprinttype   = {arXiv},
  url          = {https://www.sciencedirect.com/science/article/pii/S000187082200665X}
}

@article{CookeSamuelson2021,
  author       = {Cooke, Juliet and Samuelson, Peter},
  title        = {On the Genus Two Skein Algebra},
  journaltitle = {Journal of the London Mathematical Society},
  volume       = {104},
  number       = {5},
  year         = {2021},
  pages        = {2260--2298},
  doi          = {10.1112/jlms.12497},
  eprint       = {2008.12695},
  eprinttype   = {arXiv},
  url          = {https://londmathsoc.onlinelibrary.wiley.com/doi/10.1112/jlms.12497}
}

@article{CookeLacabanne2026,
  author       = {Cooke, Juliet and Lacabanne, Abel},
  title        = {Higher Rank {Askey--Wilson} Algebras as Skein Algebras},
  journaltitle = {Annales de l'Institut Fourier},
  volume       = {76},
  number       = {3},
  year         = {2026},
  pages        = {1055--1117},
  doi          = {10.5802/aif.3729},
  eprint       = {2205.04414},
  eprinttype   = {arXiv},
  url          = {https://aif.centre-mersenne.org/articles/10.5802/aif.3729/}
}

@online{Chen2024,
  author     = {Chen, Haimiao},
  title      = {On the Structure of {Kauffman} Bracket Skein Algebra of a
                Surface},
  year       = {2024},
  eprint     = {2406.02299},
  eprinttype = {arXiv},
  url        = {https://arxiv.org/abs/2406.02299}
}

@article{ArthamonovShakirov2019,
  author       = {Arthamonov, Semeon and Shakirov, Shamil},
  title        = {Genus Two Generalization of \(A_1\) Spherical {DAHA}},
  journaltitle = {Selecta Mathematica. New Series},
  volume       = {25},
  number       = {2},
  year         = {2019},
  eid          = {17},
  doi          = {10.1007/s00029-019-0447-1},
  eprint       = {1704.02947},
  eprinttype   = {arXiv},
  url          = {https://link.springer.com/article/10.1007/s00029-019-0447-1}
}

@article{Arthamonov2025,
  author       = {Arthamonov, Semeon},
  title        = {Classical Limit of Genus Two {DAHA}},
  journaltitle = {Selecta Mathematica. New Series},
  volume       = {31},
  number       = {1},
  year         = {2025},
  eid          = {19},
  doi          = {10.1007/s00029-024-01009-2},
  eprint       = {2309.01011},
  eprinttype   = {arXiv},
  url          = {https://link.springer.com/article/10.1007/s00029-024-01009-2}
}

@book{FarbMargalit2012,
  author    = {Farb, Benson and Margalit, Dan},
  title     = {A Primer on Mapping Class Groups},
  series    = {Princeton Mathematical Series},
  number    = {49},
  publisher = {Princeton University Press},
  location  = {Princeton, NJ},
  year      = {2012},
  isbn      = {978-0-691-14794-9},
  doi       = {10.23943/princeton/9780691147949.001.0001},
  url       = {https://academic.oup.com/princeton-scholarship-online/book/41605}
}

@article{Santharoubane2024,
  author       = {Santharoubane, Ramanujan},
  title        = {Algebraic Generators of the Skein Algebra of a Surface},
  journaltitle = {Algebraic \& Geometric Topology},
  volume       = {24},
  number       = {5},
  year         = {2024},
  pages        = {2571--2578},
  doi          = {10.2140/agt.2024.24.2571},
  url          = {https://msp.org/agt/2024/24-5/p05.xhtml}
}

@unpublished{SantharoubanePrivate2026,
  author = {Santharoubane, Ramanujan},
  title  = {Private Communication},
  date   = {2026-07},
  note   = {Email correspondence with Jin-Cheng Guu concerning the meaning
            of ``bi-ideal''}
}

@book{AskeyWilson1985,
  author    = {Askey, Richard and Wilson, James},
  title     = {Some Basic Hypergeometric Orthogonal Polynomials That
               Generalize {Jacobi} Polynomials},
  series    = {Memoirs of the American Mathematical Society},
  volume    = {54},
  number    = {319},
  publisher = {American Mathematical Society},
  year      = {1985},
  pagetotal = {55},
  isbn      = {978-1-4704-0732-2},
  doi       = {10.1090/memo/0319},
  url       = {https://pubs.ams.org/ebooks/memo/0319}
}

@book{KoekoekLeskySwarttouw2010,
  author    = {Koekoek, Roelof and Lesky, Peter A. and Swarttouw, Ren\'e F.},
  title     = {Hypergeometric Orthogonal Polynomials and Their
               \(q\)-Analogues},
  series    = {Springer Monographs in Mathematics},
  publisher = {Springer},
  location  = {Berlin},
  year      = {2010},
  pagetotal = {578},
  isbn      = {978-3-642-05013-8},
  doi       = {10.1007/978-3-642-05014-5},
  url       = {https://link.springer.com/book/10.1007/978-3-642-05014-5}
}

@article{CraneKauffmanYetter1997,
  author       = {Crane, Louis and Kauffman, Louis H. and Yetter, David N.},
  title        = {State-Sum Invariants of 4-Manifolds},
  journaltitle = {Journal of Knot Theory and Its Ramifications},
  volume       = {6},
  number       = {2},
  year         = {1997},
  pages        = {177--234},
  doi          = {10.1142/S0218216597000145},
  eprint       = {hep-th/9409167},
  eprinttype   = {arXiv},
  url          = {https://www.worldscientific.com/doi/10.1142/S0218216597000145}
}

@thesis{Tham2021,
  author      = {Tham, Ying Hong},
  title       = {On the Category of Boundary Values in the Extended
                 {Crane--Yetter} {TQFT}},
  type        = {PhD thesis},
  institution = {Stony Brook University},
  year        = {2021},
  eprint      = {2108.13467},
  eprinttype  = {arXiv},
  url         = {https://www.math.stonybrook.edu/alumni/2021-Ying-Hong-Tham.pdf}
}

@article{EtingofKirillov1994,
  author       = {Etingof, Pavel I. and Kirillov, Jr., Alexander},
  title        = {Macdonald's Polynomials and Representations of Quantum
                  Groups},
  journaltitle = {Mathematical Research Letters},
  volume       = {1},
  number       = {3},
  year         = {1994},
  pages        = {279--296},
  doi          = {10.4310/MRL.1994.v1.n3.a1},
  eprint       = {hep-th/9312103},
  eprinttype   = {arXiv},
  url          = {https://link.intlpress.com/JDetail/1806607748747915265}
}

@article{Cherednik1995,
  author       = {Cherednik, Ivan},
  title        = {Double Affine {Hecke} Algebras and {Macdonald}'s
                  Conjectures},
  journaltitle = {Annals of Mathematics},
  volume       = {141},
  number       = {1},
  year         = {1995},
  pages        = {191--216},
  doi          = {10.2307/2118632},
  url          = {https://annals.math.princeton.edu/1995/141-1/p07}
}

@article{Hikami2019,
  author       = {Hikami, Kazuhiro},
  title        = {{DAHA} and Skein Algebra of Surfaces: Double-Torus Knots},
  journaltitle = {Letters in Mathematical Physics},
  volume       = {109},
  number       = {10},
  year         = {2019},
  pages        = {2305--2358},
  doi          = {10.1007/s11005-019-01189-5},
  eprint       = {1901.02743},
  eprinttype   = {arXiv},
  url          = {https://link.springer.com/article/10.1007/s11005-019-01189-5}
}

@article{Samuelson2019,
  author       = {Samuelson, Peter},
  title        = {Iterated Torus Knots and Double Affine {Hecke} Algebras},
  journaltitle = {International Mathematics Research Notices},
  volume       = {2019},
  number       = {9},
  year         = {2019},
  pages        = {2848--2893},
  doi          = {10.1093/imrn/rnx198},
  eprint       = {1408.0483},
  eprinttype   = {arXiv},
  url          = {https://academic.oup.com/imrn/article/2019/9/2848/4103252}
}

@incollection{AganagicShakirov2012,
  author       = {Aganagic, Mina and Shakirov, Shamil},
  title        = {Refined {Chern--Simons} Theory and Knot Homology},
  booktitle    = {String-Math 2011},
  series       = {Proceedings of Symposia in Pure Mathematics},
  volume       = {85},
  publisher    = {American Mathematical Society},
  location     = {Providence, RI},
  year         = {2012},
  pages        = {3--31},
  doi          = {10.1090/pspum/085/1372},
  eprint       = {1202.2489},
  eprinttype   = {arXiv},
  url          = {https://bookstore.ams.org/PSPUM/85}
}
\end{document}